\documentclass[11pt,a4paper,twoside]{article}

\usepackage[utf8]{inputenc}
\usepackage[T1]{fontenc}

\usepackage{hyperref}
\usepackage{lmodern}
\usepackage{amssymb,amsfonts,amsthm,mathtools,mathrsfs,pifont,bigints}
\usepackage{enumitem}
\usepackage{fancyhdr}
\usepackage{ifthen}
\usepackage{caption}
\usepackage{tikz,pgfplots}

\newtheorem{theorem}{Theorem}[section]
\newtheorem{lemma}[theorem]{Lemma}
\newtheorem{prop}[theorem]{Proposition}
\newtheorem{corollary}[theorem]{Corollary}
\theoremstyle{remark}
\newtheorem{remark}[theorem]{Remark}
\theoremstyle{definition}

\usepackage{tikz}

\usetikzlibrary{calc}
\usetikzlibrary{math}
\usetikzlibrary{patterns}
\usetikzlibrary{arrows.meta,decorations.shapes}

\usepackage[left=2.5cm,right=2.5cm,bottom=2.75cm]{geometry}

\newcommand{\abs}[1]{\left| #1 \right|}
\newcommand{\norm}[1]{\left\lVert #1 \right\rVert}
\newcommand{\scp}[1]{\left\langle #1 \right\rangle}
\newcommand{\set}[1]{\left\lbrace #1\right\rbrace}
\renewcommand{\d}[1]{\,d #1}
\newcommand{\restr}[2]{\left. #1 \right|_{#2}}
\newcommand{\setcompl}[1]{{#1}^{c}}

\newcommand{\B}{\mathbb{B}}
\newcommand{\D}{\mathcal{D}}

\newcommand{\K}{\mathcal{K}}
\newcommand{\M}{\mathcal{M}}
\newcommand{\NOp}{\mathcal{N}}
\newcommand{\POp}{\mathcal{P}}

\newcommand{\AM}{\mathbf{A}}

\newcommand{\N}{\mathbb{N}}
\newcommand{\R}{\mathbb{R}}
\newcommand{\Sp}{\mathbb{S}}

\newcommand{\nd}{\partial_{\nu}}
\newcommand{\diam}[1]{\mathrm{diam}(#1)}
\newcommand{\pv}{\mathrm{p.\, v.}}

\numberwithin{equation}{section}
\numberwithin{theorem}{section}

\allowdisplaybreaks

\ifpdf
\hypersetup{
  pdftitle={Photoacoustic inversion formulas using mixed data on finite time intervals},
  pdfauthor={F. Dreier, M. Haltmeier}
}
\fi

\begin{document}
	\title{{\fontsize{12}{15}\bfseries\uppercase{Photoacoustic inversion formulas using mixed data on finite time intervals}}\thanks{\textbf{Funding}: This work has been supported by the Austrian Science Fund (FWF), project P 30747-N32.}}
	
	\author{Florian Dreier\thanks{Department of Mathematics, University of Innsbruck, Technikerstraße 13, A-6020 Innsbruck, Austria (Florian.Dreier@uibk.ac.at, Markus.Haltmeier@uibk.ac.at).}
\and and \and Markus Haltmeier\footnotemark[2]}

	\date{}
	
	\maketitle
	\begin{abstract}
  	We study the inverse source problem in photoacoustic tomography (PAT) for mixed data, which denote a weighted linear combination of the acoustic pressure and its normal derivative on an observation surface. We consider in particular the case where the data are only available on finite time intervals, which accounts for real-world usage of PAT where data are only feasible within a certain time interval. Extending our previous work, we derive explicit formulas up to a smoothing integral on convex domains with a smooth boundary, yielding exact reconstruction for circular or elliptical domains. We also present numerical reconstructions of our new exact inversion formulas on finite time intervals and compare them with the reconstructions of our previous formulas for unlimited time wave measurements.
  	

\end{abstract}

\section{Introduction}
In recent decades, PAT has attracted considerable attention in biomedical optics. This imaging method aims at recovering the spatially varying absorption coefficient of an internal source by measuring the resulting acoustic waves detected outside of the object. Here, the absorption coefficient is with respect to external electromagnetic radiation. The acoustic measurements are acquired by so-called ultrasound detectors, which lie on a surface surrounding the imaged object \cite{LiWa09,WaHu12,WaPaKuXiStWa03,XuWa05,XuWa06}. Such wave phenomena can be modeled by the $n$-dimensional standard wave equation
\begin{equation}
	\label{eq:waveeq}
	\begin{aligned}
    	\partial_t^2u(x,t)-\Delta u(x,t)&=0 &\quad &\text{for } (x,t) \in \R^n \times (0,\infty),\\
    	u(x,0)&=f(x)&\quad &\text{for } x \in \R^n,&\\
    	\partial_tu(x,0)&=g(x)&\quad &\text{for } x \in \R^n \,,
    \end{aligned}
\end{equation}
where $f$ and $g$ are the initial data and $u\colon\R^n\times[0,\infty)\to\R$ the solution of the wave equation. Extensions of the standard wave model in PAT have also been considered and are presented in  \cite{FueMooQuiFerGar21,MooReyGarGarAraGar15,UluUnlGulErk18}, for example.

In PAT, it is commonly assumed that $g$ is the zero function and $f$ has compact support in some open domain $\Omega\subset\R^n$. The function $f$ describes the optical absorption coefficient at different positions whereby various types of the biological tissues can be detected, including brain tumor and skin melanoma (see, for example, \cite{OhLiZhMaLv06,StGrSaCuCoYa10,ZhMaStWa06}). In PAT, one is interested to determine this physical quantity by using the measured pressure waves on the observation surface $\partial\Omega$.

Most analytic reconstruction methods in PAT assume that the measurements correspond to the values of the solution $u$ on $\partial\Omega\times(0,\infty)$, which is also known as the \emph{Dirichlet trace} or \emph{Dirichlet data} on $\partial\Omega\times(0,\infty)$. Depending on the measurement surface $\partial\Omega$, there exist, for example, explicit inversion formulas for the determination of $f$ that require only knowledge of $\restr{u}{\partial\Omega\times(0,\infty)}$. To name a couple of references, inversion formulas have been established for bounded smooth surfaces like spheres \cite{XuWa05,FinHalRak07,FinPatRak04,Kun07,Ngu09,NorLin81}, ellipses \cite{AnsFibMadSey12,Hal13,Hal14,Nat12,Pal14,Sal14} and also for unbounded smooth surfaces \cite{XuWa05,NorLin81,And88,Bel09,BukKar78,Faw85,HalSer15a,HalSer15b,NarRak10,nguyen2014reconstruction} including planar, quadric hypersurfaces and cylindrical surfaces. In \cite{Pal14,DoKun18,Kun11,Kun15}, non-smooth and finite open measurement surfaces have been considered.

As pointed out in \cite{Fin05,XuWa04}, for example, the measured data from piezoelectric transducers is generally a complex combination of the acoustic field $u$ and its normal gradient (normal derivative) of $u$ on the observation surface $\partial\Omega$ and therefore, corresponding to the output $\Phi(\restr{u}{S\times(0,\infty)},\nd u)$ of some function $\Phi$. Theoretical results on the problem of recovering $f$ from the normal derivative of $u$ (which is also referred to as the \emph{Neumann trace} of $u$) are rather rare in the literature. In \cite{Fin05,FinRak07}, an inversion formula for a smooth function $g$ from the normal gradient of $u$ in \eqref{eq:waveeq} with initial data $(0,g)$ on spheres in 3D is presented. Recently, in \cite{ZanMooHal19}, a series formula for spheres and in \cite{DreHal20,DreHal21} an exact inversion formula of back-projection type for ellipsoids in arbitrary dimensions has been derived.

Accounting for more realistic measurement models, which additionally depends on the normal derivative, we use a linearization of $\Phi$, resulting in data \[u_{a,b}(x,t)\coloneqq au(x,t)+b\nd u(x,t),\quad (x,t)\in\partial\Omega\times (0,\infty),\] for some weights $a,b\in\R$. Throughout this article, we refer to these measurements as \emph{mixed data} or \emph{mixed trace}. Recovering the absorption coefficient from mixed data has previously been studied in \cite{ZanMooHal19} for spheres in arbitrary dimensions and in \cite{DreHal20} for circles in two dimensions. Another mathematical model for piezoelectric sensors is proposed, for example, in \cite{Aco20}.

\subsection{Inversion from finite time intervals}
In this paper, we study the inverse problem in PAT, where the acoustic measurements of the photoacoustic sensors are given only on a finite time interval. Specifically, we assume a finite end time $T\in(0,\infty)$ for measuring the acoustic signals and consider mixed data on the time interval $(0,T)$ as an output signal of the transducers. As opposed to the previous articles \cite{DreHal20,DreHal21}, where we assumed that Neumann traces in arbitrary dimensions as well as mixed data in two dimensions are available on the whole time interval $(0,\infty)$, this justifies truncating data at finite values of $T$. The problem of recovering the initial data $(f,0)$ from Dirichlet traces on a finite time interval has already been discussed in \cite{FinHalRak07}, where an inversion formula in dimension two has been developed. More precisely, they used Abel transform inversion for recovering the spherical means with centers lying on a circle in $\R^2$ from its Dirichlet traces. As a consequence, the values of the solution of the wave equation for time points greater than the diameter of the circle can be recovered by knowing only the values for time points smaller than the diameter, which has been rearranged as single finite time inversion formula. In this article, we use a similar approach and derive relations between spherical means and Dirichlet traces as well as the \emph{weighted spherical means} (which will be defined later) and Neumann traces in even dimensions. In case of odd dimensions, such additional considerations are not necessary, as we will see in the next subsection.
\subsection{Previous work}
The subsequent calculations are based on our previous work in \cite{DreHal20,DreHal21}, where we derived the reconstruction formulas
\begin{equation}
	\label{eq:invneumanneven}
	f(x)=\frac{1}{2^{\frac{n-2}{2}}\pi^{\frac{n}{2}}}(-1)^{\frac{n-2}{2}}\int_{\partial\Omega}\int_{\norm{x-y}}^\infty \frac{\left(\partial_t t^{-1} \right)^{\frac{n-2}{2}}\nd u(y,t)}{\sqrt{t^2-\norm{x-y}^2}}\d{t}\d{\sigma(y)}-\K_\Omega f(x),
\end{equation}
and
\begin{equation}
	\label{eq:invneumannodd}
	f(x)=\frac{1}{(2\pi)^{\frac{n-1}{2}}}(-1)^{\frac{n-3}{2}}\int_{\partial\Omega} \left(\frac{1}{t}\partial_t\right)^{\frac{n-3}{2}}\left(\frac{1}{t}\nd u\right)(y,\norm{x-y})\d{\sigma(y)}-\K_\Omega f(x),
\end{equation}
over unbounded time intervals for Neumann traces, holding true in even and odd dimensions and for smooth functions $f\in C_c^\infty(\Omega)$ with compact support in an open convex domains $\Omega\subset\R^n$ with smooth boundary.
For the Dirichlet data case, we make use of the explicit inversion formula (see \cite{Hal14})
\begin{equation}
	\label{eq:invdirichleteven}
	f(x)=\frac{1}{2^{\frac{n-2}{2}}\pi^{\frac{n}{2}}}(-1)^{\frac{n-2}{2}}\nabla_x\cdot\int_{\partial\Omega}\nu(y)\int_{\norm{x-y}}^{\infty}\frac{\left(\partial_t t^{-1}\right)^{\frac{n-2}{2}}u(y,t)}{\sqrt{t^2-\norm{x-y}^2}}\d{t}\d{\sigma(y)}+\K_{\Omega} f(x),
\end{equation}
being valid in even dimensions under the assumptions presented above. Existing formulas for Dirichlet traces in odd dimensions such as in \cite{Hal14,FinRak07}, for example, as well as formula \eqref{eq:invneumannodd} for Neumann traces only require finite time measurements, because they are equal to zero on the boundary $\partial\Omega$ for all time points greater than or equal to the diameter of $\Omega$. For that reason, we only investigate finite time inversion formulas for mixed traces in odd dimensions.

At this point, we remark that $\K_\Omega f=0$ for circular or elliptical domains \cite{Hal14,DreHal21}. Therefore, in this article whenever an explicit formula for $f$ depends on $\K_\Omega f$, the formula is exact for circular or elliptical domains.

Formulas \eqref{eq:invneumanneven}--\eqref{eq:invdirichleteven}, which are used in the derivation of our new formulas, are based on two integral identities for the standard wave equation.
These integral identities include the free-space Green function twice (two boundary integrals) and are a kind of Kirchhoff integral formula. Formulas \eqref{eq:invneumanneven}--\eqref{eq:invdirichleteven} are derived by expressing one boundary integral via the other. Further methods for inverting the initial pressure are, for example, so-called far-field approximations (see \cite{BurMatHalPal07}, for example), where the pressure field is approximated by Radon projections instead of spherical means of the initial pressure, or the delay and sum (DAS) algorithm, which is a simple beamforming algorithm \cite{MaPeYuaCheXuWanCar19}.
\subsection{Outline}
The paper is organized as follows. First, we start with some notations being used throughout the article. Then, we derive our main results in section \ref{sec:timeboundinveven}, where we derive explicit inversion formulas for Neumann and Dirichlet traces on convex domains $\Omega\subset\R^n$ with a smooth boundary as well as for mixed traces on circular domains over finite intervals $(0,T)$. Our restriction on the end time is that $T$ has to be greater than or equal to the diameter of the domain $\Omega$. Note that although only knowledge of the measurements on the time interval from zero to the diameter of $\Omega$ are theoretically necessary, the numerical results in section \ref{sec:numresults} do not show the same numerical reconstructions for different end times. We will also compare the numerical results of our new inversion formulas with numerical results of formulas \eqref{eq:invneumanneven} and \eqref{eq:invdirichleteven} for unbounded time intervals. Section \ref{sec:timeboundinvodd} studies the inversion problem in odd dimensions. Before ending our article in section \ref{sec:appendix} with some auxiliary results, we will give a short conclusion in section \ref{sec:conclusion}.

\section{Notation}
Let $\Omega\subset\R^n$ be an open set, $x\in \Omega$ and $f\colon \Omega\to\R$ differentiable at $x$. For a vector $v\in\R^n$ we denote by $D_v f(x)$ the directional derivative of $f$ at $x$ along the vector $v$, that is, for a sufficient small $\varepsilon>0$, the derivative of the function \[(-\varepsilon,\varepsilon)\to\R\colon t\mapsto f(x+tv)\] at zero. If $\Omega$ is a bounded domain with smooth boundary and $x\in\partial\Omega$, we use the abbreviation $\nd f(x)\coloneqq D_{\nu(x)}f(x)$, indicating the \emph{normal derivative} of $f$ at $x$, where $\nu\colon\partial\Omega\to\R^n$ is the outward unit normal vector field of $\Omega$. For a function $f\colon\R^n\times[0,\infty)\to\R$, we also use the notation $D_v f(x,t)$ and $\nd f(x,t)$  to denote the directional derivative and normal derivative of $f$ at $(x,t)\in \R^n\times[0,\infty)$ with respect to the spatial variable. Note that the chain rule gives
\begin{equation}
	\label{eq:reldirder}
	D_v f(x,t)=\scp{v,\nabla f(x,t)},\quad (x,t)\in\R^n\times[0,\infty),
\end{equation}
where $\nabla f(x,t)$ denotes the gradient of $f$ in the spatial variable.

The spherical mean operator of an integrable function $f\colon\R^n\to\R$ is defined as
\begin{equation}
	\label{eq:defspmeanop}
	\M f\colon \R^n\times [0,\infty)\to \R\colon (x,r)\mapsto \frac{1}{n\omega_n}\int_{\Sp^{n-1}} f(x+ry) \d{\sigma(y)},
\end{equation}
where $\omega_n$ denotes the volume of $n$-dimensional unit ball, $\Sp^{n-1}$ the $n-1$-dimensional unit sphere in $\R^n$ and $\sigma$ the standard volume measure on $\Sp^{n-1}$ (without any angular weighting). Moreover, for brevity we write \[\M_\nu f \colon \R^n\times [0,\infty)\to \R\colon (x,r)\mapsto \nd\M f(x,r),\] denoting the normal derivative of $\M f$ in the spatial variable or the \emph{weighted spherical mean transform}.

For any function $f\colon (a,c)\cup(c,b)\to\R$ with $a<b$ and $c\in(a,b)$, which is integrable on $(a,c-\varepsilon)\cup(c+\varepsilon,b)$ for all $0<\varepsilon<\min\set{c-a,b-c}$, we write
	\[\pv\int_a^b f(x)\d{x}\coloneqq \lim_{\varepsilon\searrow 0}\left(\int_a^{c-\varepsilon} f(x)\d{x} + \int_{c+\varepsilon}^b f(x)\d{x}\right),\]
provided the above limit exists. This form of integral is known as the \emph{Cauchy principle value}.

Lastly, we use $\POp$ to denote the operator which takes a function $g\in C_c^\infty(\B^n)$ with compact support in the open unit ball in $\R^n$ to the restriction of the solution of wave equation \eqref{eq:waveeq} with initial data $(0,g)$ on $\Sp^{n-1}\times[0,\infty)$. Moreover, we use the symbol $\NOp$ to indicate the operator which maps a function $f\in C_c^\infty(\B^n)$ to the restriction of $t^{n-2}\M f$ on $\Sp^{n-1}\times[0,\infty)$.

From now on, we assume that $\Omega\subset\R^n$ is a convex domain with smooth boundary and $T\geq\diam{\Omega}\coloneqq\sup\set{\norm{x-y}\mid x,y\in\Omega}$. Furthermore, for the rest of this article, we denote by $u\colon\R^n\times[0,\infty)\to\R$ the solution of the wave equation with initial data $(f,0)$, for a given function  $f\in C_c^\infty(\Omega)$.
\section{Inversion in even dimensions}
\label{sec:timeboundinveven}
In this section, we present our main results for the even-dimensional case. The following theorems are based on the relations given in Proposition \ref{prop:dvsolwaveeqeven}, \ref{prop:solabelinteqwaveeqeven} and the reconstruction formulas derived from them in Theorem \ref{thm:invdvspmeanopeven} for inverting the weighted spherical mean transform. We will also see that explicit inversion formulas depend on a measureable kernel $k_T\colon (0,T)^2\to\R$, which is independent of the spatial dimension $n\geq 2$.

\subsection{Inversion from Neumann data on finite time intervals}
The first theorem presents an explicit inversion formula for determining the initial data of the wave equation from Neumann traces on the bounded manifold $\partial\Omega\times (0,T)$.
\begin{theorem}
	\label{thm:invbtneumanneven}
	Let $n\geq 2$ be an even number, $f\in C_c^\infty(\Omega)$ be a smooth function with compact support in $\Omega$ and $k_T\colon (0,T)^2\to\R$ be the kernel function defined by $k_T(r_1,r_2)\coloneqq\frac{2}{\pi\sqrt{\abs{r_1^2-r_2^2}}}\tilde{k}_T(r_1,r_2)$ for $r_1\neq r_2$, where
	\begin{align*}
		\tilde{k}_T(r_1,r_2)\coloneqq\begin{cases}
			\frac{1}{2}\log\left(\frac{\sqrt{T^2-r_2^2}-\sqrt{r_1^2-r_2^2}}{\sqrt{T^2-r_2^2}+\sqrt{r_1^2-r_2^2}}\right),\quad &r_1>r_2,\\
			\arctan\left(\frac{\sqrt{T^2-r_1^2}}{\sqrt{r_2^2-r_1^2}}\right),\quad& r_1<r_2,\\
		\end{cases}
	\end{align*}
	and $k_T(r_1,r_2)\coloneqq 0$ for $r_1=r_2$. Then, for every $x\in\Omega$, we have
	\begin{equation}
		\label{eq:invneumannbteven}
		f(x)=\frac{2(-1)^{\frac{n-2}{2}}}{\omega_n\gamma_n}\int_{\partial\Omega}\int_0^{T}k_T(\norm{x-y},t)\left(\partial_t t^{-1}\right)^{\frac{n-2}{2}}\nd u(y,t)\d{t}\d{\sigma(y)}-\K_{\Omega} f(x),
	\end{equation}
	where $\omega_n$ denotes the volume of the $n$-dimensional unit ball and $\gamma_n=2\cdot 4\cdots (n-2)\cdot n$.
\end{theorem}
The following two identities, which are explicit inversion formulas for the weighted spherical mean transform, are essential for the derivation of our reconstruction formula in Theorem \ref{thm:invbtneumanneven}.
\begin{theorem}
	\label{thm:invdvspmeanopeven}
	For $f\in C_c^\infty(\Omega)$ and $x\in\Omega$, the relations
	\begin{equation}
		\label{eq:invndspop1}
		\begin{aligned}
		f(x)=&\frac{2n(-1)^{\frac{n-1}{2}}}{\omega_n\gamma_n^2}\int_{\partial\Omega}\int_0^{T}\left(\partial_r r\left(r^{-1}\partial_r\right)^{n-2} r^{n-2}\M_{\nu}f(y,r)\right)\frac{\log\left(\frac{r+\norm{x-y}}{\abs{r-\norm{x-y}}}\right)}{2\norm{x-y}}\d{r}\d{\sigma(y)}\\
		&- \K_{\Omega} f(x)
		\end{aligned}
	\end{equation}
	and
	\begin{equation}
		\label{eq:invndspop2}
			f(x)=\frac{2n(-1)^{\frac{n-2}{2}}}{\omega_n\gamma_n^2}\int_{\partial\Omega}\pv\int_0^{T}\frac{r\left(r^{-1}\partial_r\right)^{n-2} r^{n-2}\M_{\nu}f(y,r)}{r^2-\norm{x-y}^2}\d{r}\d{\sigma(y)}- \K_{\Omega} f(x).
	\end{equation}
	in even dimensions hold.
\end{theorem}
\begin{proof}
	\begin{enumerate}[wide=\parindent,label=(\roman*)]
		\item Inserting relation \eqref{eq:dvsolwaveeqeven2} into to the reconstruction formula \eqref{eq:invneumanneven} lead to
		\begin{equation}
			\label{eq:invndspopproof1}
			f(x)=\frac{2n}{\gamma_n^2\omega_n}(-1)^{\frac{n-2}{2}}\int_{\partial\Omega}\int_{\norm{x-y}}^\infty\int_0^t \frac{r\partial_r r\left(r^{-1}\partial_r\right)^{n-2}r^{n-2}\M_{\nu} f(y,r)}{t\sqrt{t^2-\norm{x-y}^2}\sqrt{t^2-r^2}}\d{r}\d{t}\d{\sigma(y)}-\K_\Omega f(x).
		\end{equation}
		Since $\norm{x-y}<t$ and $0<r<\min\set{T,t}$ for $(t,r)\in(0,\infty)^2$ if and only if $0<r<T$ and $\max\set{r,\norm{x-y}}<t$, we obtain from Fubini's theorem
		\begin{align*}
			\int_{\partial\Omega}\int_{\norm{x-y}}^\infty&\int_0^r \frac{\abs{r\partial_r r\left(r^{-1}\partial_r\right)^{n-2}r^{n-2}\M_{\nu} f(y,r)}}{t\sqrt{t^2-\norm{x-y}^2}\sqrt{t^2-r^2}}\d{r}\d{t}\d{\sigma(y)}\\
			&=\int_{\partial\Omega}\int_0^{T}\int_{\max\set{\norm{x-y},r}}^\infty \frac{\abs{r\partial_r r\left(r^{-1}\partial_r\right)^{n-2}r^{n-2}\M_{\nu} f(y,r)}}{t\sqrt{t^2-\norm{x-y}^2}\sqrt{t^2-r^2}}\d{t}\d{r}\d{\sigma(y)}\\
		&\leq C\int_{\partial\Omega}\int_0^{T}\int_{\max\set{\norm{x-y},r}}^\infty \frac{r}{t\sqrt{t^2-\norm{x-y}^2}\sqrt{t^2-r^2}}\d{t}\d{r}\d{\sigma(y)},
		\end{align*}
		where $C\coloneqq \sup\set{\left.\abs{\partial_r r\left(r^{-1}\partial_r\right)^{n-2}r^{n-2}\M_{\nu} f(y,r)}\,\right\vert (y,r)\in\partial\Omega\times(0,T)}$. Next, from Lemma \ref{lem:intinvsphericalmeanop} we see that the above triple integral equals \[\int_{\partial\Omega}\int_0^{T}\frac{1}{2\norm{x-y}}\log\left(\frac{r+\norm{x-y}}{\abs{r-\norm{x-y}}}\right)\d{r}\d{\sigma(y)}.\] Then, the monotonicity and rules of the logarithmic-function yield the boundedness of the above triple integral. Hence, the application of Fubini's theorem on the triple integral in \eqref{eq:invndspopproof1} and the same calculations as before yield the first identity.
		\item Treating the inner integral in the principle value sense, applying integration by parts and using that $r\left(r^{-1}\partial_r\right)^{n-2}r^{n-2}\M_{\nu} f(y,\cdot)$ has compact support in $(0,T)$ for $y\in\partial\Omega$, we observe that the inner integral on right-hand side in \eqref{eq:invndspop1} is equal to the limit of
		\begin{multline*}
			\left.\frac{r\left(r^{-1}\partial_r\right)^{n-2}r^{n-2}\M_{\nu} f(y,r)}{2c}\log\left(\frac{r+c}{\varepsilon}\right)\right\vert_{r=c-\varepsilon}\\
			-\left.\frac{r\left(r^{-1}\partial_r\right)^{n-2}r^{n-2}\M_{\nu} f(y,r)}{2c}\log\left(\frac{r+c}{\varepsilon}\right)\right\vert_{r=c+\varepsilon}\\
			+\int_{(0,T)\setminus(c-\varepsilon,c+\varepsilon)}\frac{r\left(r^{-1}\partial_r\right)^{n-2}r^{n-2}\M_{\nu} f(y,r)}{r^2-c^2}\d{r}
		\end{multline*}
		as $\varepsilon\searrow 0$, where we set for brevity $c\coloneqq \norm{x-y}$. Expanding the logarithm and applying the mean value theorem, the boundary term can be estimated by the sum of the two terms on right-hand side in
		\begin{multline*}
			\frac{1}{2c}\Bigg|\left.r\left(r^{-1}\partial_r\right)^{n-2}r^{n-2}\M_{\nu} f(y,r)\log(r+c)\right\vert_{r=c-\varepsilon}\\
			-\left.r\left(r^{-1}\partial_r\right)^{n-2}r^{n-2}\M_{\nu} f(y,r)\log(r+c)\right\vert_{r=c+\varepsilon}\Bigg|\\
			\leq\frac{\varepsilon}{c}\sup_{r\in[c-\delta,c+\delta]}\abs{\partial_rr\left(r^{-1}\partial_r\right)^{n-2}r^{n-2}\M_{\nu} f(y,r)\log(r+c)}
		\end{multline*}
		and
		\begin{multline*}
			\frac{\abs{\log(\varepsilon)}}{2c}\abs{\left.r\left(r^{-1}\partial_r\right)^{n-2}r^{n-2}\M_{\nu} f(y,r)\right\vert_{r=c-\varepsilon}-\left.r\left(r^{-1}\partial_r\right)^{n-2}r^{n-2}\M_{\nu} f(y,r)\right\vert_{r=c+\varepsilon}}\\
			\leq\frac{\abs{\varepsilon\log(\varepsilon)}}{c}\sup_{r\in[c-\delta,c+\delta]}\abs{\partial_rr\left(r^{-1}\partial_r\right)^{n-2}r^{n-2}\M_{\nu} f(y,r)},
		\end{multline*}
		for sufficient small $\delta>0$. Hence, letting $\varepsilon\searrow 0$ shows that the limit of the boundary terms equals zero and therefore, equation \eqref{eq:invndspop2} holds.\qedhere
	\end{enumerate}
\end{proof}

\begin{proof}[Proof of Theorem \ref{thm:invbtneumanneven}]
For better readability, we divide the proof into several parts. 
\begin{enumerate}[wide=\parindent,label=(\roman*)]
	\item Inserting relation \eqref{eq:dvsolabelinteqwaveeq2} into the inner integral on the right-hand side in \eqref{eq:invndspop2}, using the definition of the integral in the principle value sense and applying Fubini's theorem lead to
\begin{align*}
	\frac{2\gamma_n}{\pi n}\lim_{\varepsilon\searrow 0}\int_{(0,T)\setminus(c-\varepsilon,c+\varepsilon)}&\int_0^r \frac{r\left(\partial_tt^{-1}\right)^{\frac{n-2}{2}}\nd u(y,t)}{(r^2-c^2)\sqrt{r^2-t^2}}\d{t}\d{r}\\
	=\frac{2\gamma_n}{\pi n}\lim_{\varepsilon\searrow 0}\Bigg(&\int_0^{c-\varepsilon}\int_{I_{\varepsilon,t}}\frac{r\left(\partial_tt^{-1}\right)^{\frac{n-2}{2}}\nd u(y,t)}{(r^2-c^2)\sqrt{r^2-t^2}}\d{r}\d{t}\\
	&+\int_{c-\varepsilon}^c\int_{c+\varepsilon}^{T}\frac{r\left(\partial_tt^{-1}\right)^{\frac{n-2}{2}}\nd u(y,t)}{(r^2-c^2)\sqrt{r^2-t^2}}\d{r}\d{t}\\
	&+\int_c^{T}\int_{\max\set{t,c+\varepsilon}}^{T}\frac{r\left(\partial_tt^{-1}\right)^{\frac{n-2}{2}}\nd u(y,t)}{(r^2-c^2)\sqrt{r^2-t^2}}\d{r}\d{t}\Bigg),
\end{align*}
where we used the abbreviations $c\coloneqq\norm{x-y}$ and $I_{\varepsilon,t}\coloneqq (t,c-\varepsilon]\cup[c+\varepsilon,T)$ for fixed $y\in\partial\Omega$.
	\item In the next step, we compute above inner integrals separately. The first inner integral can be split up into two integrals which, by Lemma \ref{lem:intinvbtneumanneven}, can be evaluated to \[ f_{1,\varepsilon}(t)\coloneqq\int_t^{c-\varepsilon}\frac{r}{(r^2-c^2)\sqrt{r^2-t^2}}\d{r}=\frac{1}{2\sqrt{c^2-t^2}}\log\left(\frac{\sqrt{c^2-t^2}-\sqrt{(c-\varepsilon)^2-t^2}}{\sqrt{c^2-t^2}+\sqrt{(c-\varepsilon)^2-t^2}}\right)\] and
\begin{multline*}
	f_{2,\varepsilon}(t)\coloneqq\int_{c+\varepsilon}^{T}\frac{r}{(r^2-c^2)\sqrt{r^2-t^2}}\d{r}\\
	=\frac{1}{2\sqrt{c^2-t^2}}\log\left(\frac{\sqrt{T^2-t^2}-\sqrt{c^2-t^2}}{\sqrt{T^2-t^2}+\sqrt{c^2-t^2}}\frac{\sqrt{(c+\varepsilon)^2-t^2}+\sqrt{c^2-t^2}}{\sqrt{(c+\varepsilon)^2-t^2}-\sqrt{c^2-t^2}}\right).
\end{multline*}
Again by Lemma \ref{lem:intinvbtneumanneven}, we see that the second inner integral also corresponds to \[f_{3,\varepsilon}(t)\coloneqq \frac{1}{2\sqrt{c^2-t^2}}\log\left(\frac{\sqrt{T^2-t^2}-\sqrt{c^2-t^2}}{\sqrt{T^2-t^2}+\sqrt{c^2-t^2}}\frac{\sqrt{(c+\varepsilon)^2-t^2}+\sqrt{c^2-t^2}}{\sqrt{(c+\varepsilon)^2-t^2}-\sqrt{c^2-t^2}}\right),\] whereas the last inner integral is equal to \[f_{4,\varepsilon}(t)\coloneqq\frac{1}{\sqrt{t^2-c^2}}\left(\arctan\left(\frac{\sqrt{T^2-c^2}}{\sqrt{t^2-c^2}}\right)-\arctan\left(\frac{\sqrt{{\max\set{t,c+\varepsilon}}^2-c^2}}{\sqrt{t^2-c^2}}\right)\right).\]
\item In the last step of the proof, we show
\begin{multline*}
	\lim_{\varepsilon\searrow 0}\int_0^{c-\varepsilon} (f_{1,\varepsilon}(t)+f_{2,\varepsilon}(t))\left(\partial_tt^{-1}\right)^{\frac{n-2}{2}}\nd u(y,t)\d{t}\\
	=\int_0^c \log\left(\frac{\sqrt{T^2-t^2}-\sqrt{c^2-t^2}}{\sqrt{T^2-t^2}+\sqrt{c^2-t^2}}\right)\frac{\left(\partial_tt^{-1}\right)^{\frac{n-2}{2}}\nd u(y,t)}{2\sqrt{c^2-t^2}}\d{t},
\end{multline*}
\begin{equation*}
	\lim_{\varepsilon\searrow 0}\int_{c-\varepsilon}^c f_{3,\varepsilon}(t)\left(\partial_tt^{-1}\right)^{\frac{n-2}{2}}\nd u(y,t)\d{t}=0
\end{equation*}
and
\begin{equation*}
	\lim_{\varepsilon\searrow 0}\int_c^{T} f_{4,\varepsilon}(t)\left(\partial_tt^{-1}\right)^{\frac{n-2}{2}}\nd u(y,t)\d{t}=\int_c^{T}\arctan\left(\frac{\sqrt{T^2-c^2}}{\sqrt{t^2-c^2}}\right)\frac{\left(\partial_tt^{-1}\right)^{\frac{n-2}{2}}\nd u(y,t)}{\sqrt{t^2-c^2}}\d{t},
\end{equation*}
which yield the claimed identity.

First, we write
\begin{multline*}
	 f_{1,\varepsilon}(t)+f_{2,\varepsilon}(t)=\frac{1}{2\sqrt{c^2-t^2}}\Bigg[\log\left(\frac{\sqrt{T^2-t^2}-\sqrt{c^2-t^2}}{\sqrt{T^2-t^2}+\sqrt{c^2-t^2}}\right)\\
	 +\log\left(\frac{\sqrt{c^2-t^2}-\sqrt{(c-\varepsilon)^2-t^2}}{\sqrt{c^2-t^2}+\sqrt{(c-\varepsilon)^2-t^2}}\frac{\sqrt{(c+\varepsilon)^2-t^2}+\sqrt{c^2-t^2}}{\sqrt{(c+\varepsilon)^2-t^2}-\sqrt{c^2-t^2}}\right)\Bigg].
\end{multline*}
Then, from Lemma \ref{lem:fsinvbtneumanneven} we have
\begin{multline*}
	\int_0^{c-\varepsilon}\abs{f_{1,\varepsilon}(t)+f_{2,\varepsilon}(t)}\abs{\left(\partial_tt^{-1}\right)^{\frac{n-2}{2}}\nd u(y,t)}\d{t}\\
	\leq C_1\int_0^c\frac{1}{2\sqrt{c^2-t^2}}\left(\abs{\log\left(\frac{\sqrt{T^2-t^2}-\sqrt{c^2-t^2}}{\sqrt{T^2-t^2}+\sqrt{c^2-t^2}}\right)}+\log(6+\sqrt{2})\right)\d{t},
\end{multline*}
where $C_1\coloneqq \sup\set{\abs{\left.\left(\partial_tt^{-1}\right)^{\frac{n-2}{2}}\nd u(y,t)}\,\right\vert t\in(0,\infty)}$. Hence, Lebesgue's theorem and the second statement in Lemma \ref{lem:intinvbtneumanneven} show the first identity.\\The second limit can be seen to vanish by expanding the logarithm similarly and using the estimate
\begin{multline*}
	\int_{c-\varepsilon}^c\abs{f_{3,\varepsilon}(t)}\abs{\left(\partial_tt^{-1}\right)^{\frac{n-2}{2}}\nd u(y,t)}\d{t}\leq C_1(C_2+\log(6+\sqrt{2}))\int_{c-\varepsilon}^c \frac{1}{2\sqrt{c^2-t^2}}\d{t}\\
	=\frac{C_1(C_2+\log(6+\sqrt{2}))}{2c}\left(\frac{\pi}{2}-\arcsin\left(\frac{c-\varepsilon}{c}\right)\right),
\end{multline*}
where $C_2\coloneqq\sup\set{\left.\abs{\log\left(\frac{\sqrt{T^2-t^2}-\sqrt{c^2-t^2}}{\sqrt{T^2-t^2}+\sqrt{c^2-t^2}}\right)}\,\right\vert t\in[0,c]}$.\\Finally, since $\arctan$ is bounded by $\pi/2$, applying Lebesgue's theorem again on the third limit yield the final statement.\qedhere
\end{enumerate}
\end{proof}
\subsection{Inversion from Dirichlet data on finite time intervals}
Now, we consider the inverse problem for Dirichlet data $u$ on $\partial\Omega\times(0,T)$. By using the back-projection formula for the spherical mean operator in \cite{Hal14} and the ideas for Neumann traces from the previous subsection, we deduce the following result. 
\begin{theorem}
	\label{thm:invdirichletbteven}
	Let $n\geq 2$ be an even number, $f\in C_c^\infty(\Omega)$ be a smooth function with compact support in $\Omega$ and $k_T\colon (0,T)^2\to\R$ be the kernel function as defined in Theorem \ref{thm:invbtneumanneven}. Then, for every $x\in\Omega$, we have
	\begin{equation}
		\label{eq:invdirichletbteven}
		f(x)=\frac{2(-1)^{\frac{n-2}{2}}}{\omega_n\gamma_n}\nabla_x\cdot\int_{\partial\Omega}\nu(y)\int_0^{T} k_T(\norm{x-y},t)\left(\partial_t t^{-1}\right)^{\frac{n-2}{2}}u(y,t)\d{t}\d{\sigma(y)}
		+\K_{\Omega} f(x),
	\end{equation}
	where $\omega_n$ denotes the volume of the $n$-dimensional unit ball and $\gamma_n=2\cdot 4\cdots (n-2)\cdot n$.
\end{theorem}
\begin{proof}
	In \cite{Hal14}, it has been shown that
	\begin{equation*}
			f(x)=\frac{2n(-1)^{\frac{n-2}{2}}}{\omega_n\gamma_n^2}\nabla_x\cdot\int_{\partial\Omega}\nu(y)\ \pv\int_0^{\diam{\Omega}}\frac{r\left(r^{-1}\partial_r\right)^{n-2} r^{n-2}\M f(y,r)}{r^2-\norm{x-y}^2}\d{r}\d{\sigma(y)}+ \K_{\Omega} f(x).
	\end{equation*}
	As we observe in the above formula, the inner integral has the same form to that one in \eqref{eq:invndspop2}. Thus, by inserting \eqref{eq:solabelinteqwaveeq2} into the above equation and using the same arguments as in the proof of Theorem \ref{thm:invbtneumanneven}, we immediately obtain \eqref{eq:invdirichletbteven}.\qedhere
\end{proof}
\subsection{Inversion from mixed data on finite time intervals}
\label{sec:timeboundinvevenmixed}
In this section, we consider the problem of determining the initial data of the wave equation from measurements of the type $u_{a,b}$ on the boundary of open balls $\B_\rho^n(z)$ with radius $\rho>0$ and center $z\in\R^n$. The main result is as follows.
\begin{theorem}
	Let $n\geq 2$ be an even number, $\B_{\rho}^n(z)\subset\R^n$ the open ball with radius $\rho>0$ and center $z\in\R^n$, $f\in C_c^\infty(\B_{\rho}^n(z))$, $a,b\in\R$ with $b\neq 0$ and $k_T\colon (0,T)^2\to\R$ be the kernel function as defined in Theorem \ref{thm:invbtneumanneven}. Then, for every $x\in\B_{\rho}^n(z)$, we have
	\begin{equation}
		\label{eq:invformmixedeven}
		f(x)=\frac{2(-1)^{\frac{n-2}{2}}}{b\omega_n\gamma_n}\int_{\partial\B_{\rho}^n(z)}\int_{\norm{x-y}}^{\infty}\frac{\left(\partial_t t^{-1}\right)^{\frac{n-2}{2}}u_{a,b}(y,t)}{\sqrt{t^2-\norm{x-y}^2}}\d{t}\d{\sigma(y)}
	\end{equation}
	and
	\begin{equation}
		\label{eq:invformmixedbteven}
		f(x)=\frac{2(-1)^{\frac{n-2}{2}}}{b\omega_n\gamma_n}\int_{\partial\B_{\rho}^n(z)}\int_0^{T}k_T(\norm{x-y},t)\left(\partial_t t^{-1}\right)^{\frac{n-2}{2}}u_{a,b}(y,t)\d{t}\d{\sigma(y)},
	\end{equation}
	where $\omega_n$ denotes the volume of the $n$-dimensional unit ball and $\gamma_n=2\cdot 4\cdots (n-2)\cdot n$.
\end{theorem}
For the proof, we first derive formula \eqref{eq:invformmixedeven}, which requires measurements for every time point $t>0$. Similarly as in the previous sections, we use then the derived exact inversion formula for unbounded time intervals to establish formula \eqref{eq:invformmixedbteven} that requires only measurements on the finite time interval $(0,T)$.

For the derivation of the first result, we make use of the inversion formulas in even dimensions (see \cite{FinHalRak07})
\begin{align}
	f(x)&=-2(\POp^*t\partial_t^2\POp f)(x)\label{eq:invformadj1}\\
	f(x)&=-2(\POp^*\partial_t t\partial_t\POp f)(x)\label{eq:invformadj2}
\end{align}
for recovering $f\in C_c^\infty(\B^n)$ from $\POp f$. Here, $\POp^{*}$ denotes the formal adjoint operator of $\POp$. We use an analytic expression of the adjoint $\POp^{*}$ for certain functions $F\colon \Sp^{n-1}\times[0,\infty)$ from which we are able to deduce a range condition for the solution of the wave equation with initial data $(f,0)$ presented in Lemma \ref{lem:rangecondeven}.

Note that in \cite{FinHalRak07}, there has been derived an representation of $\POp^{*}$ in the two-dimensional case for continuous functions $F\colon \Sp^{n-1}\times[0,\infty)$ with sufficient small decay, i.e., $F(y,t)=\mathcal{O}(t^{-\alpha})$ as $t\to\infty$ for some $\alpha>0$. For higher dimensions $n>2$ , we additionally assume that $F$ is $(n-2)/2$ times continuously differentiable on $(0,\infty)$ in the second component, $(t^{-1}\partial_t)^kt^{-1}F(y,t)=\mathcal{O}(t^{-\alpha})$ for $0\leq k\leq (n-4)/2$ and $(\partial_t t^{-1})^{(n-2)/2}F(y,t)=\mathcal{O}(t^{-\alpha})$ as $t\to\infty$. Then, the adjoint of $\POp$ can be expressed as (see \cite{FinRak07})
\begin{equation}
	\label{eq:adjpeven}
	\POp^* F(x)=\frac{(-1)^{\frac{n-2}{2}}}{\gamma_n\omega_n}\int_{\Sp^{n-1}}\int_{\norm{x-y}}^\infty \frac{(\partial_t t^{-1})^{(n-2)/2}F(y,t)}{\sqrt{t^2-\norm{x-y}^2}}\d{t}\d{\sigma(y)},\quad x\in\R^n.	
\end{equation}	                                   
\begin{lemma}
	\label{lem:rangecondeven}
	Let $n\geq 2$ be an even number, $\B_{\rho}^n(z)\subset\R^n$ the open ball with radius $\rho>0$ and center $z\in\R^n$ and $f\in C_c^\infty(\B_{\rho}^n(z))$. Then, for every $x\in\B_{\rho}^n(z)$, we have
	\begin{equation}
		\label{eq:rangecondeven}
		0=\int_{\partial\B_{\rho}^n(z)}\int_{\norm{x-y}}^\infty \frac{(\partial_t t^{-1})^{(n-2)/2}u(y,t)}{\sqrt{t^2-\norm{x-y}^2}}\d{t}\d{\sigma(y)}.
	\end{equation}
\end{lemma}
\begin{proof}
	In the following, we assume without loss of generality that $z=0$ and $\rho=1$. The remaining statement follows from a translation and rescaling argument. Note that from the product rule we have $\partial_t t\partial_t\POp=\partial_t\POp f+t\partial_t^2\POp f$ and therefore $\partial_t\POp f=\partial_tt\partial_t\POp f-t\partial_t^2\POp f$. Hence, subtracting \eqref{eq:invformadj2} from \eqref{eq:invformadj1} gives \[0=\POp^*(\partial_t t\partial_t\POp f-t\partial_t^2\POp f)(x)=(\POp^*\partial_t\POp f)(x).\] Since $\partial_t\POp f$ fulfills the conditions for the analytic expression of $\POp^*$ in \eqref{eq:adjpeven} (see, for example, \eqref{eq:solwaveeqeven} and \cite[Lemma 3.4]{DreHal21}), we deduce \[0=\int_{\Sp^{n-1}}\int_{\norm{x-y}}^\infty \frac{(\partial_t t^{-1})^{(n-2)/2}\partial_t\POp f(y,t)}{\sqrt{t^2-\norm{x-y}^2}}\d{t}\d{\sigma(y)}.\] From \eqref{eq:waveeq} we easily see that $\partial_t\POp f$ solves the wave equation with initial data $(f,0)$ and therefore, \eqref{eq:rangecondeven} is proved.\qedhere
\end{proof}
\begin{proof}[Proof of formula \eqref{eq:invformmixedeven}]
	From \eqref{eq:invneumanneven} we have
	\begin{equation}
		\label{eq:prinvformmixedeven}
		f(x)=\frac{2(-1)^{\frac{n-2}{2}}}{b\omega_n\gamma_n}\int_{\partial\B_{\rho}^n(z)}\int_{\norm{x-y}}^{\infty}\frac{\left(\partial_t t^{-1}\right)^{\frac{n-2}{2}}b\nd u(y,t)}{\sqrt{t^2-\norm{x-y}^2}}\d{t}\d{\sigma(y)}.
	\end{equation}
	Furthermore, \eqref{eq:rangecondeven} implies
	\begin{equation}
		\label{eq:prrangecondeven}
		0=\frac{2(-1)^{\frac{n-2}{2}}}{b\omega_n\gamma_n}\int_{\partial\B_{\rho}^n(z)}\int_{\norm{x-y}}^\infty \frac{(\partial_t t^{-1})^{(n-2)/2}au(y,t)}{\sqrt{t^2-\norm{x-y}^2}}\d{t}\d{\sigma(y)}.
	\end{equation}
	Hence, adding \eqref{eq:prinvformmixedeven} and \eqref{eq:prrangecondeven} leads to \eqref{eq:invformmixedeven}. 
\end{proof}
As a consequence of Lemma \eqref{lem:rangecondeven}, we obtain the following range conditions for the spherical mean operator in even dimensions.
\begin{lemma}
	Let $f\in C_c^\infty(\B_{\rho}^n(z))$ be a smooth function with compact support in the open ball with radius $\rho>0$ and center $z\in\R^n$. Then, for every $x\in\B_{\rho}^n(z)$, the range conditions
	\begin{equation}
		\label{eq:rangecondspopeven1}
		\begin{aligned}
		0=\int_{\partial\B_{\rho}^n(z)}\int_0^{T}\left(\partial_r r\left(r^{-1}\partial_r\right)^{n-2} r^{n-2}\M f(y,r)\right)\frac{\log\left(\frac{r+\norm{x-y}}{\abs{r-\norm{x-y}}}\right)}{2\norm{x-y}}\d{r}\d{\sigma(y)}
		\end{aligned}
	\end{equation}
	and
	\begin{equation}
		\label{eq:rangecondspopeven2}
		\begin{aligned}
			0=\int_{\partial\B_{\rho}^n(z)}\pv\int_0^{T}\frac{r\left(r^{-1}\partial_r\right)^{n-2} r^{n-2}\M f(y,r)}{r^2-\norm{x-y}^2}\d{r}\d{\sigma(y)}
		 \end{aligned}
	\end{equation}
	for the spherical mean operator in even dimensions hold.
\end{lemma}
\begin{proof}
	The statements follows from inserting formula \eqref{eq:solwaveeqeven3} into \eqref{eq:rangecondeven} and analogous calculations as in the proof of Theorem \ref{thm:invdvspmeanopeven}.
\end{proof}
The last lemma in this section presents a range condition for the solution of the wave equation on the bounded manifold $\partial\B_{\rho}^n(z)\times(0,T)$, which is the key ingredient for the derivation of the second formula in Theorem \ref{thm:invbtneumanneven}.
\begin{lemma}
	\label{lem:rangecondbteven}
	Let $n\geq 2$ be an even number, $\B_{\rho}^n(z)\subset\R^n$ the open ball with radius $\rho>0$ and center $z\in\R^n$ and $f\in C_c^\infty(\B_{\rho}^n(z))$. Then, for every $x\in\B_{\rho}^n(z)$, we have
	\begin{equation}
		\label{eq:rangecondbteven}
		0=\int_{\partial\B_{\rho}^n(z)}\int_0^{T}k_T(\norm{x-y},t)\left(\partial_t t^{-1}\right)^{\frac{n-2}{2}} u(y,t)\d{t}\d{\sigma(y)}.\\
	\end{equation}
\end{lemma}
\begin{proof}
	Similarly to the proof of Theorem \ref{thm:invbtneumanneven}, we observe that inner integral in the range condition of \eqref{eq:rangecondspopeven2} has the same form as in \eqref{eq:invndspop2}. Thus, by inserting the relation for the spherical mean transform \eqref{eq:solabelinteqwaveeq2} into \eqref{eq:rangecondspopeven2} and using the same arguments as in the proof for Neumann traces, the right-hand side in \eqref{eq:rangecondspopeven2} can be transformed to the double integral in \eqref{eq:rangecondbteven}.
\end{proof}
\begin{proof}[Proof of formula \eqref{eq:invformmixedeven}]
	Again, this follows from an addition of \eqref{eq:invneumannbteven} and \eqref{eq:rangecondbteven}.
\end{proof}
\section{Inversion in odd dimensions}
\label{sec:timeboundinvodd}
Subsequently to the previous section, where the even-dimensional case has been studied, we now present new results for recovering the initial data of the wave equation from measurements on finite time intervals in odd dimensions. As already mentioned at the beginning of this article, we only study the inverse problem for mixed traces, since existing filtered backprojection formulas for Neumann and Dirichlet traces require only finite time measurements. Again, we consider open balls $\B_\rho^n(z)$ and their boundary as their measurement surface.
\subsection{Inversion from mixed data on finite time intervals}
The inverse problem of determining $f$ from the wave data $u_{a,b}$ on the boundary of open balls $\B_\rho^n(z)$ between the time points zero and $2\rho$ can be solved similarly to the even-dimensional case. The main statement reads as follows.
\begin{theorem}
	\label{thm:invformmixedbtodd}
	Let $n\geq 3$ be an odd number, $\B_{\rho}^n(z)\subset\R^n$ the open ball with radius $\rho>0$ and center $z\in\R^n$, $f\in C_c^\infty(\B_{\rho}^n(z))$ and $a,b\in\R$ with $b\neq 0$. Then, for every $x\in\B_{\rho}^n(z)$, we have
	\begin{equation}
		\label{eq:invformmixedbtodd}
		f(x)=\frac{2(-1)^{\frac{n-3}{2}}}{bn\gamma_n\omega_n}\int_{\partial\B_{\rho}^n(z)} \left(\left(t^{-1}\partial_t\right)^{\frac{n-3}{2}}t^{-1}\right)u_{a,b}(y,\norm{x-y})\d{\sigma(y)},
	\end{equation}
	where $\omega_n$ denotes the volume of the $n$-dimensional unit ball and $\gamma_n=1\cdot 3\cdots (n-2)$.
\end{theorem}
The proof of Theorem \ref{thm:invformmixedbtodd} is based on the reconstruction formulas (see \cite{FinRak07})
\begin{align}
	f(x)&=c_n(\NOp^*t\D^{(n-3)/2}t^{-1}\partial_t^2t\POp f)(x)\label{eq:invformadjodd1}\\
	f(x)&=c_n(\NOp^*t\D^{(n-3)/2}t^{-1}\partial_t t\partial_t\POp f)(x),\label{eq:invformadjodd2}
\end{align}
for $x\in\B^n$, which are valid in odd dimensions for recovering $f\in C_c^\infty(\B^n)$ from $\POp f$. Here, $\NOp^*$ denotes the formal $L^2$ adjoint of $\NOp$, $\D\coloneqq\frac{1}{2t}\partial_t$ and $c_n\coloneqq\frac{(-1)^{(n-1)/2}\sqrt{\pi}}{\Gamma(n/2)}$. In \cite{FinPatRak04}, there has been derived the explicit expression
\begin{equation}
	\label{eq:adjn}
	\NOp^*F(x)=\frac{1}{n\omega_n}\int_{\Sp^{n-1}}\frac{F(y,\norm{x-y})}{\norm{x-y}}\d{\sigma(y)},\quad x\in\R^n,
\end{equation}
provided that $F\colon\Sp^{n-1}\times[0,\infty)\to\R$ is smooth and zero to infinite order in the time variable at $t=0$. Equation \eqref{eq:adjn} and reconstruction formulas \eqref{eq:invformadjodd1}, \eqref{eq:invformadjodd2} can be used to deduce the following range condition for the Dirichlet trace in odd dimensions, which looks slightly different to that one in even dimensions.
\begin{lemma}
	\label{lem:rangecondbtodd}
	Let $n\geq 3$ be an odd number, $\B_{\rho}^n(z)\subset\R^n$ the open ball with radius $\rho>0$ and center $z\in\R^n$ and $f\in C_c^\infty(\B_{\rho}^n(z))$. Then, for every $x\in\B_{\rho}^n(z)$, we have
	\begin{equation}
		\label{eq:rangecondbtodd}
		0=\int_{\partial\B_{\rho}^n(z)}\left(\left(t^{-1}\partial_t\right)^{\frac{n-3}{2}}t^{-1}\right)u(y,\norm{x-y})\d{\sigma(y)}.
	\end{equation}
\end{lemma}
\begin{proof}
	As in the even-dimensional case, we deduce from the product rule the relation $\partial_t\POp f=\partial_t^2t\POp f-\partial_tt\partial_t\POp f$. Then, subtracting \eqref{eq:invformadjodd2} from \eqref{eq:invformadjodd1} gives
	\begin{align*}
		0&=\NOp^{*}(t\D^{(n-3)/2}t^{-1}\partial_t^2t\POp f-t\D^{(n-3)/2}t^{-1}\partial_t t\partial_t\POp f)(x)\\
		&=\NOp^{*}(t\D^{(n-3)/2}t^{-1}\partial_t\POp f)(x).
	\end{align*}
	Finally, using the explicit expression of \eqref{eq:adjn} leads to
	\begin{align*}
		0&=\int_{\partial\B_{\rho}^n(z)} \frac{(t\D^{(n-3)/2}t^{-1}\partial_t\POp f)(y,\norm{x-y})}{\norm{x-y}}\d{\sigma(y)}\\
		&=\int_{\partial\B_{\rho}^n(z)}\left(\left(t^{-1}\partial_t\right)^{\frac{n-3}{2}}t^{-1}\right)u(y,\norm{x-y})\d{\sigma(y)},
	\end{align*}
	where we used again in the last step that $\partial_t\POp f$ is the unique solution of \eqref{eq:waveeq} with initial data $(f,0)$.
\end{proof}
\begin{proof}[Proof of Theorem \ref{thm:invformmixedbtodd}]
	In \cite[Theorem 4.1]{DreHal21}, it has been shown \[f(x)=\frac{2(-1)^{\frac{n-3}{2}}}{bn\gamma_n\omega_n}\int_{\partial\B_{\rho}^n(z)}\left(\left(t^{-1}\partial_t\right)^{\frac{n-3}{2}}t^{-1}\right)b\nd u(y,\norm{x-y})\d{\sigma(y)}.\] Then, using \eqref{eq:rangecondbtodd} leads to \[f(x)=\frac{2(-1)^{\frac{n-3}{2}}}{bn\gamma_n\omega_n}\int_{\partial\B_{\rho}^n(z)}\left(\left(t^{-1}\partial_t\right)^{(n-3)/2}t^{-1}\right)(au+b\nd u)(y,\norm{x-y})\d{\sigma(y)},\]
	which shows the final result.
\end{proof}
\section{Numerical implementation and experiments}
\label{sec:numresults}
Following the theoretical results in sections \ref{sec:timeboundinveven} and \ref{sec:timeboundinvodd}, we now present numerical implementations in two dimensions of our new inversion formulas for wave measurements on finite time intervals and compare them with old formulas requiring unlimited time wave measurements. We will consider all three types of traces. For the sake of completeness, we briefly discuss how we discretized initial data and the corresponding wave measurements and elaborate how we numerically implemented our new inversion formulas. Throughout this section, we suppose that $f$ has compact support in some ball $\B_\rho^2(z)\subset\R^2$ with radius $\rho>0$ and center $z\in\R^2$.
\subsection{Discretization of simulated data and numerical implementation of new inversion formulas}
Let $N$ denote the number of samples in one direction, $\Delta x=\frac{2\rho}{N-1}$ the step size and $T\geq 2\rho$ the end time. Our goal is to determine discrete values of $f$ on the uniform grid $\set{-\rho+i\Delta x\mid i\in\set{0,\ldots,N-1}}^2$ from discrete wave data $\mathbf{m}_T[k,l]$ given on the points of the circle $$y_k\coloneqq z+\rho(\cos(\varphi_k),\sin(\varphi_k)),\quad k=0,\ldots,N_\varphi-1$$ with $N_\varphi\coloneqq\lceil2\rho\pi/\Delta x\rceil$, $\varphi_k\coloneqq k\Delta\varphi$ and $\Delta\varphi\coloneqq 2\rho\pi/N_\varphi$, and time points $$t_l=l\Delta t,\quad l=0,\ldots,N_t-1,$$ where $0<\Delta t<T$ and $N_t\coloneqq \lfloor T/\Delta t+1\rfloor$. The measurements $\mathbf{m}_T[k,l]$ correspond either to the discrete values of a weighted weighted Dirichlet trace or a weighted Neumann trace or a mixed trace
\begin{equation*}
	\mathbf{d}_T[k,l]\coloneqq a u(y_k,t_l),\quad \mathbf{n}_T[k,l]\coloneqq b\nd u(y_k,t_l),\quad \mathbf{mix}_T[k,l]\coloneqq a u(y_k,t_l)+b\nd u(y_k,t_l),
\end{equation*}
at $(y_k,t_l)$ for some weights $a,b\neq 0$. For the numerical simulation, we computed them by solving the wave equation with the fast Fourier transform (FFT) on the whole grid and using interpolation to obtain the corresponding values on the detector points $y_k$.

\subsubsection*{Implementation of formula for Neumann/mixed traces}
First, we give details on the numerical realization of formula \eqref{eq:invneumannbteven}. Formula \eqref{eq:invformmixedbteven} can be implemented analogously. Now, given the discrete values of the Neumann trace $\mathbf{n}_T[k,l]=\nd u(y_k,t_l)$, we approximate \eqref{eq:invneumannbteven} by
\begin{align*}
	f(x_{i,j})&=\frac{\rho}{\pi}\int_0^{2\pi}\int_0^T k_T(\norm{x_{i,j}-h(\varphi)},t)\nd u(h(\varphi),t)\d{t}\d{\varphi}\\
		&\approx \frac{\rho\Delta \varphi}{\pi}\sum_{k=0}^{N_\varphi-1}\int_0^T k_T(\norm{x_{i,j}-y_k},t)\nd u(y_k,t)\d{t},
\end{align*}
where $h\colon(0,2\pi)\to\R^2\colon \varphi\mapsto z+\rho(\cos(\varphi),\sin(\varphi))^T$ denotes a parametrization of $\partial\B_\rho^2(z)$ and $x_{i,j}=(-R+i\Delta x,-R+j\Delta x)^T$ a point on the grid for some $0\leq i,j\leq N-1$. In order to approximate the above inner integral, we first compute
\begin{align*}
	\int_0^T k_T(t_l,t)\nd u(y_k,t)\d{t}&\approx\sum_{m=1}^{N_t-1}\int_{t_{m-1}}^{t_m}k_T(t_l,t)\nd u(y_k,t)\d{t}\\
	&\approx\frac{2}{\pi}\sum_{m=1}^{N_t-1}\int_{t_{m-1}}^{t_m}\frac{t_m}{\sqrt{\abs{t_m^2-t^2}}}t_m^{-1}\tilde{k}_T(t_l,t_m)\nd u(y_k,t_m)\d{t}\\
	&=\frac{2}{\pi}\sum_{m=1}^{N_t-1}\abs{\sqrt{\abs{t_m^2-t_l^2}}-\sqrt{\abs{t_{m-1}^2-t_l^2}}}t_m^{-1}\tilde{k}_T(t_l,t_m)\mathbf{n}_T[k,m]
\end{align*}
for $0\leq l \leq N_t-2$. Defining $\AM_{T}[l,m]\coloneqq\abs{\sqrt{\abs{t_m^2-t_l^2}}-\sqrt{\abs{t_{m-1}^2-t_l^2}}}t_m^{-1}\tilde{k}_{T}(t_l,t_m)$ and $\AM_{T}[l,0]=0$ for $0\leq l\leq N_t-2$ and $1\leq m\leq N_t-1$, we finally use
\begin{equation*}
	\mathbf{F}_T(\mathbf{n}_T)[i,j]\coloneqq\frac{2\rho\Delta \varphi}{\pi^2}\sum_{k=0}^{N_\varphi-1}\mathrm{interpolate}\left(\left(\mathbf{n}_T*\AM_{T}'\right)[k,:],\norm{x_{i,j}-y_k}\right)
\end{equation*}
as a discretization of formula \eqref{eq:invneumannbteven}, where the inner term expresses the interpolated value of $\norm{x_{i,j}-y_k}$ respectively the array $(\mathbf{n}_T*\AM_{T}')[k,:]$ and time points $t_l$ for $l=0,\ldots,N_t-2$. Here, $\AM_{T}'$ denotes the transpose of the matrix $\AM_{T}$. Formula \eqref{eq:invneumanneven} for unbounded time intervals can be discretized in a similar way
\begin{equation*}
	\mathbf{F}_{\infty}(\mathbf{n}_T)[i,j]\coloneqq\frac{\rho\Delta \varphi}{\pi}\sum_{k=0}^{N_\varphi-1}\mathrm{interpolate}\left(\left(\mathbf{n}_T*\AM'\right)[k,:],\norm{x_{i,j}-y_k}\right)
\end{equation*}
by using the matrix $\AM\in\R^{N_t\times N_t}$ defined by \[\AM[i,j]\coloneqq \begin{cases}\sqrt{t_j^2-t_i^2}-\sqrt{t_{j-1}^2-t_i^2},\quad &i<j,\\0,\quad&\text{else,}\end{cases}\] and cutting off the inner integral in \eqref{eq:invneumanneven} from $T$ to infinity. As can be observed in the definition of $\mathbf{F}_T(\mathbf{n}_T)$  and  $\mathbf{F}_{\infty}(\mathbf{n}_T)$, beside the factor $2/\pi$ the main difference between these two formulas is the kernel function $\tilde{k}_T$, which appears in the definition of $\mathbf{F}_T$ in $\mathbf{A}_T$, whereas in the definition of $\mathbf{F}_\infty$ no kernel function in $\mathbf{A}$ is being used.

The corresponding discretized versions of formulas \eqref{eq:invformmixedeven} and \eqref{eq:invformmixedbteven} are denoted by $\mathbf{F}_{\infty}(\mathbf{mix}_T)$ and $\mathbf{F}_{T}(\mathbf{mix}_T)$.
\subsubsection*{Implementation of formula for Dirichlet traces}
From \eqref{eq:invdirichletbteven} we see that implementing the formula for Dirichlet data requires the numerical computation of the double integral for each component. Writing $\nu(y_k)=(\cos(\varphi_k),\sin(\varphi_k))^T$ for $0\leq k\leq N_\varphi-1$, the double integrals can be numerically evaluated to
\begin{equation*}
	\mathbf{G}_{T}(\mathbf{d}_T)^1[i,j]\coloneqq\sum_{k=0}^{N_\varphi-1}\cos(\varphi_k)\mathrm{interpolate}\left(\left(\mathbf{d}_T*\AM_{T}'\right)[k,:],\norm{x_{i,j}-y_k}\right)
\end{equation*}
\begin{equation*}
	\mathbf{G}_{T}(\mathbf{d}_T)^2[i,j]\coloneqq\sum_{k=0}^{N_\varphi-1}\sin(\varphi_k)\mathrm{interpolate}\left(\left(\mathbf{d}_T*\AM_{T}'\right)[k,:],\norm{x_{i,j}-y_k}\right),
\end{equation*}
resulting in the discretized formula
\begin{equation*}
	\mathbf{G}_{T}(\mathbf{d}_T)[i,j]\coloneqq\frac{2\rho\Delta\varphi}{\pi^2}\left(CD^1(\mathbf{G}_{T}(\mathbf{d}_T)^1)[i,j]+CD^2(\mathbf{G}_{T}(\mathbf{d}_T)^2)[i,j]\right),
\end{equation*}
where $CD^1$ and $CD^1$ denote the central differences in $x$ and $y$ direction, respectively. The disretization $\mathbf{G}_{\infty}(\mathbf{d}_T)$ of formula \eqref{eq:invdirichleteven} can be approximated similarly by cutting off the inner integral from $T$ to infinity, using $\AM$ instead of $\AM_{T}$ and changing the coefficient to $\rho\Delta\varphi/\pi$.
\subsection{A numerical experiment}
In the following, we perform our numerical computations on a $[-1,1]\times[-1,1]$ grid with $N\coloneqq 257$ coordinate points. As initial data of the wave equation we use the phantom presented in Figure \ref{fig:phantom}, which has been similarly created as the phantom in \cite[Figure 1]{DreHal20}, and denoted it by $\mathbf{F}\in\R^{257\times 257}$. Note that the phantom has compact support in the open unit ball and the entries of $\mathbf{F}$ lie in the interval $[0,1]$.
\begin{figure}
	\centering
	\includegraphics[scale=0.25]{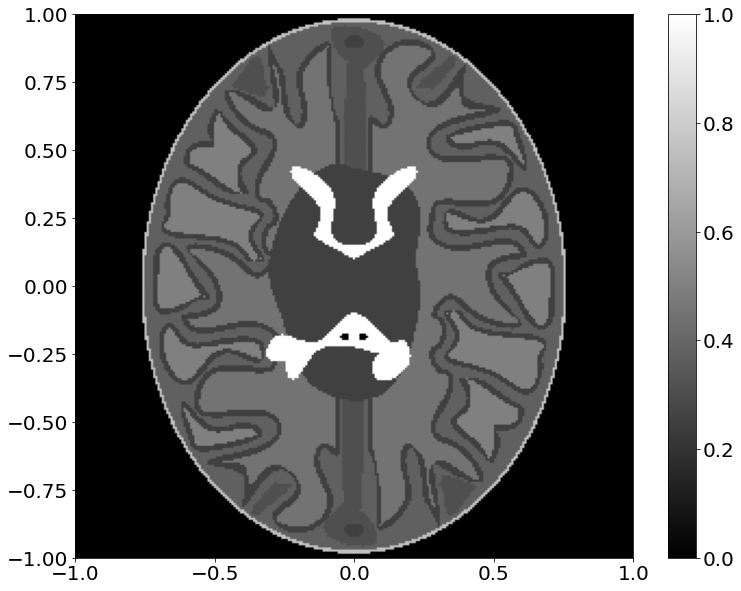}
	\caption{Initial data of our numerical experiment.}
	\label{fig:phantom}
\end{figure}
We therefore assume that $\Sp^1$ is the detection surface and $\Delta t=10^{-4}$ the time difference in which the simulated acoustic waves are measured. The end time is $T=2$ (=$\diam{\B^2}$), the shortest time for which our inversion formulas yield theoretically exact reconstruction. Figure \ref{fig:wavedata} shows the different wave measurements of the head phantom with and without Gaussian noise added to the data. Here, $20\%$ Gaussian noise added data means that normally distributed data with standard deviation $0.2$ of the maximum was added to the original data. The two weights are set to $a=1$ and $b=1/10$. Note that we have chosen the weights such that discrete values of the Neumann trace and the Dirichlet trace are approximately equally balanced in the sense that both parts significantly contribute to the frequency response of the mixed trace. Such a behavior has been actually observed in practical applications, for example for data of piezoelectric detectors often employed in PAT \cite{PalHarKovNus17}.
\begin{figure}
	\begin{minipage}{0.32\textwidth}
		\includegraphics[scale=0.2]{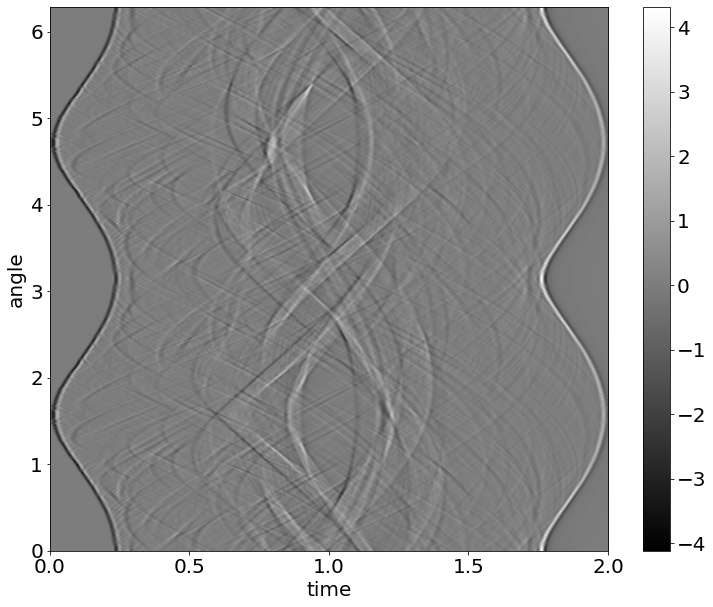}
	\end{minipage}
	\begin{minipage}{0.32\textwidth}
		\includegraphics[scale=0.2]{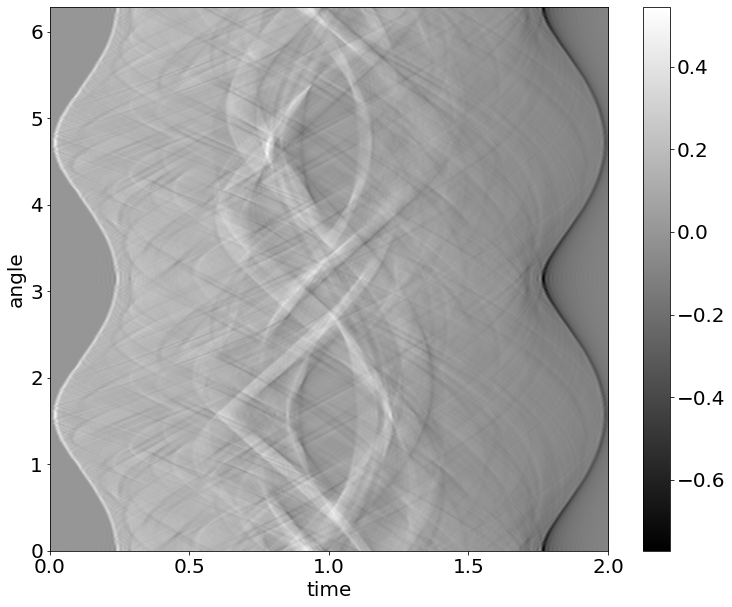}
	\end{minipage}
	\begin{minipage}{0.32\textwidth}
		\includegraphics[scale=0.2]{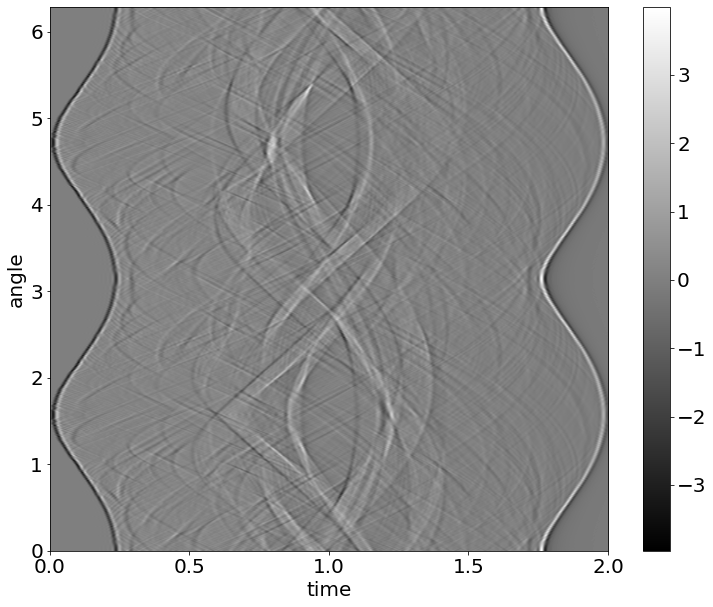}
	\end{minipage}\\
	\begin{minipage}{0.32\textwidth}
		\includegraphics[scale=0.2]{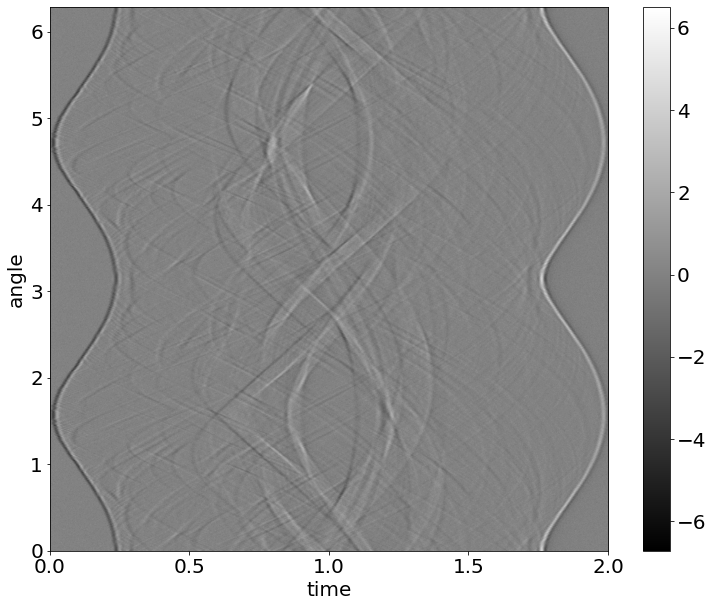}
	\end{minipage}
	\begin{minipage}{0.32\textwidth}
		\includegraphics[scale=0.2]{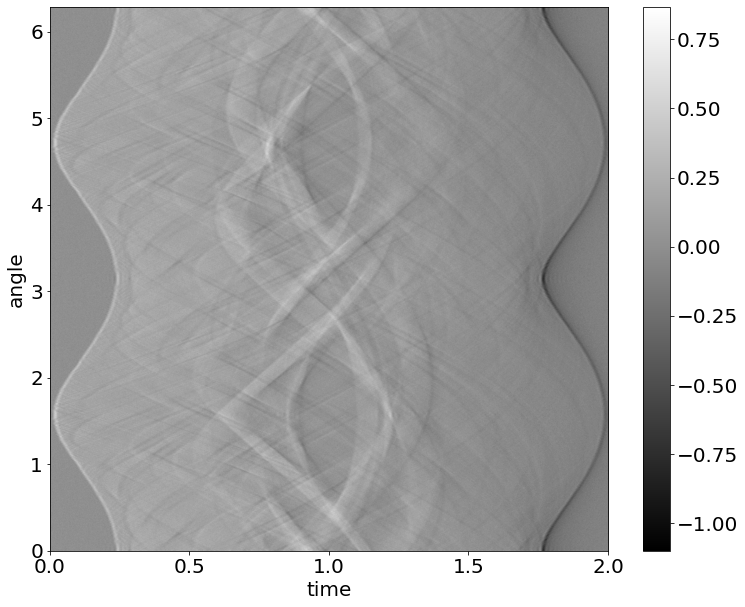}
	\end{minipage}
	\begin{minipage}{0.32\textwidth}
		\includegraphics[scale=0.2]{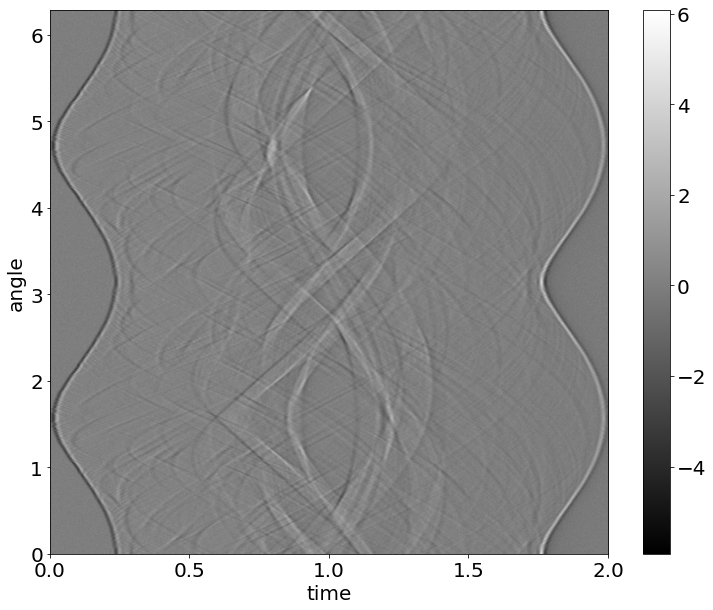}
	\end{minipage}
	\caption{Simulated wave data. Top (from left to right): Neumann data $\mathbf{n}_T$, Dirichlet data $\mathbf{d}_T$ and mixed data $\mathbf{mix}_T$. Below (from right to left): Neumann data $\mathbf{n}_T$, Dirichlet data $\mathbf{d}_T$ and mixed data $\mathbf{mix}_T$ with $20\%$ Gaussian noise.}
	\label{fig:wavedata}
\end{figure}

Figure \ref{fig:recexact} shows the numerical approximations of our previous filtered backprojection formulas in the left column ($\mathbf{F}_{\infty}(\mathbf{n}_T)$ and $\mathbf{G}_{\infty}(\mathbf{d}_T)$) and our new inversion formulas over finite time intervals for noisy-free Neumann and Dirichlet data in the right column ($\mathbf{F}_T(\mathbf{n}_T)$ and $\mathbf{G}_{T}(\mathbf{d}_T)$). Due to the cut off of the improper integrals in \eqref{eq:invneumanneven} and \eqref{eq:invdirichleteven} from $T$ to infinity, we can detect a slight error inside the detection surface which doesn't belong to the phantom itself. The gray value outside of the phantom in the left reconstruction from Neumann data has approximately the constant value $0.12$, whereas the right reconstruction with our new inversion formula hardly shows any error. A similar behavior can be observed in the reconstructions from Dirichlet traces: The gray value on the left image below has here an approximate value of $-0.05$ and the right image below has values in the range of thousands.
\begin{figure}[h!]
	\begin{minipage}{0.5\textwidth}
		\includegraphics[scale=0.25]{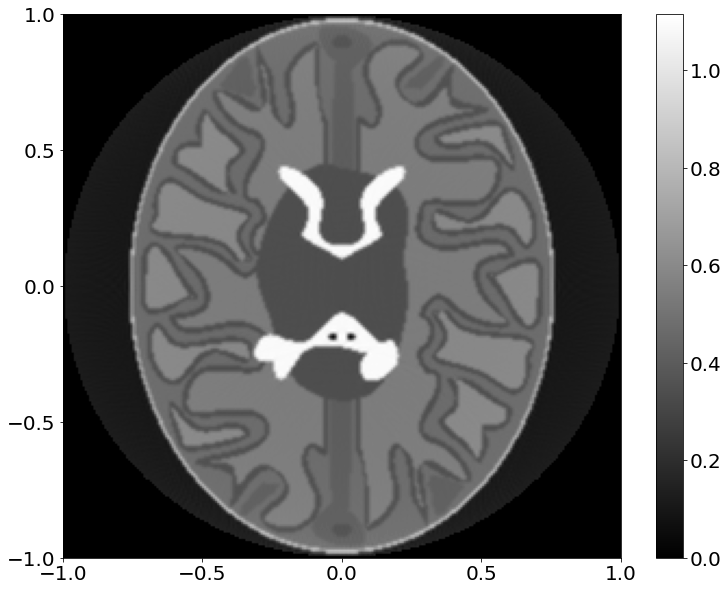}
	\end{minipage}
	\begin{minipage}{0.5\textwidth}
		\includegraphics[scale=0.25]{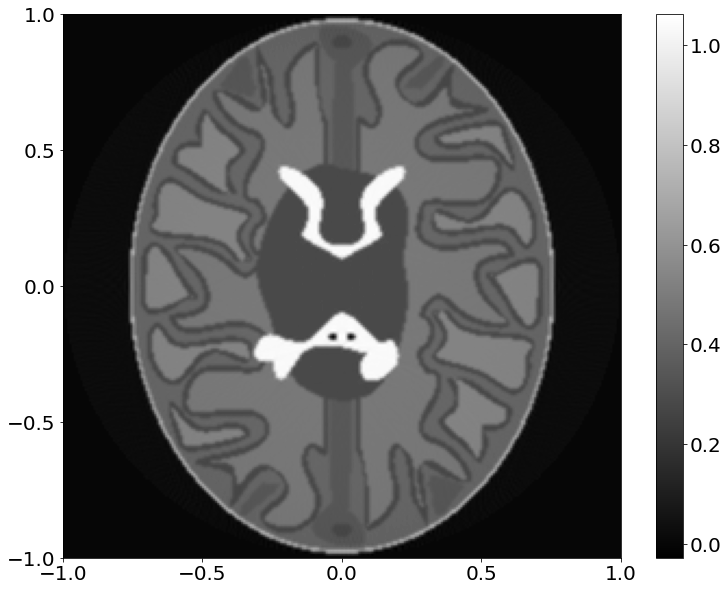}
	\end{minipage}
	\begin{minipage}{0.5\textwidth}
		\includegraphics[scale=0.25]{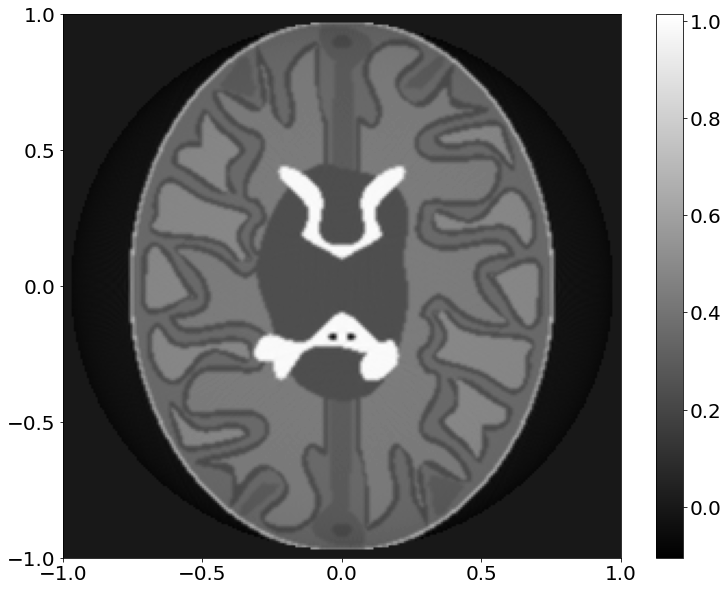}
	\end{minipage}
	\begin{minipage}{0.5\textwidth}
		\includegraphics[scale=0.25]{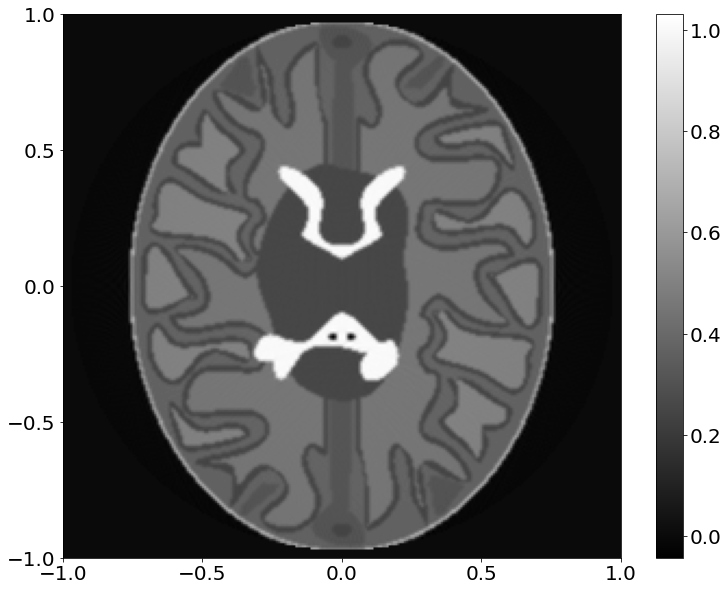}
	\end{minipage}
	\caption{Reconstructions with exact data. Top, left: $\mathbf{F}_{\infty}(\mathbf{n}_T)$. Top, right: $\mathbf{F}_T(\mathbf{n}_T)$. Below, left: $\mathbf{G}_{\infty}(\mathbf{d}_T)$. Below, right: $\mathbf{G}_{T}(\mathbf{d}_T)$.}
	\label{fig:recexact}
\end{figure}

In order to study the stability of our inversion formulas, we also applied the discretized formulas on noisy data. Figure \ref{fig:recnoisy20} shows the obtained reconstructions using $20\%$ Gaussian noise added data.
\begin{figure}[h!]
	\begin{minipage}{0.5\textwidth}
		\includegraphics[scale=0.25]{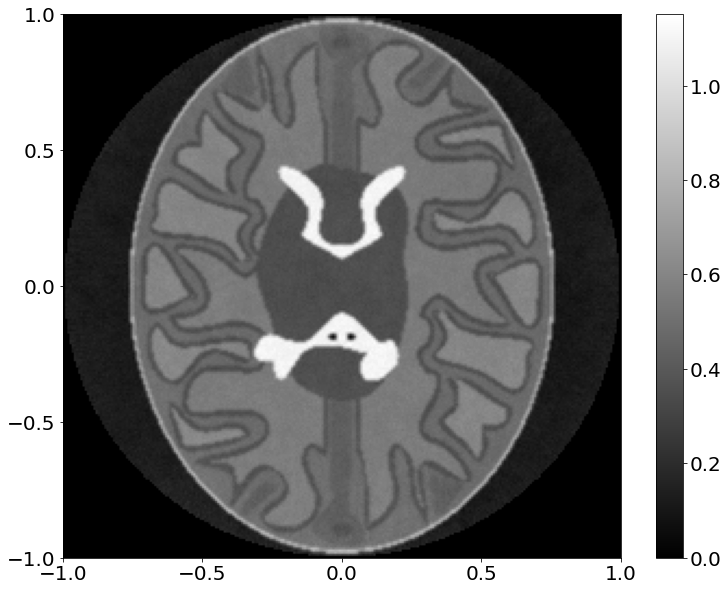}
	\end{minipage}
	\begin{minipage}{0.5\textwidth}
		\includegraphics[scale=0.25]{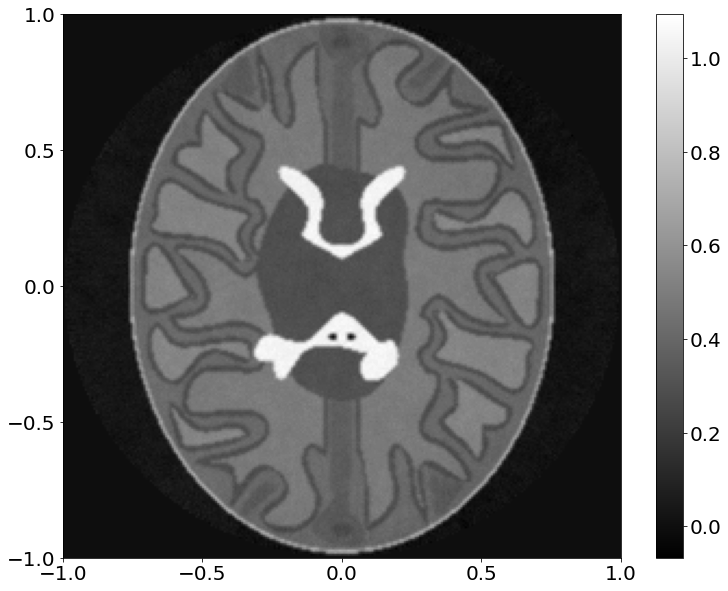}
	\end{minipage}
	\begin{minipage}{0.5\textwidth}
		\includegraphics[scale=0.25]{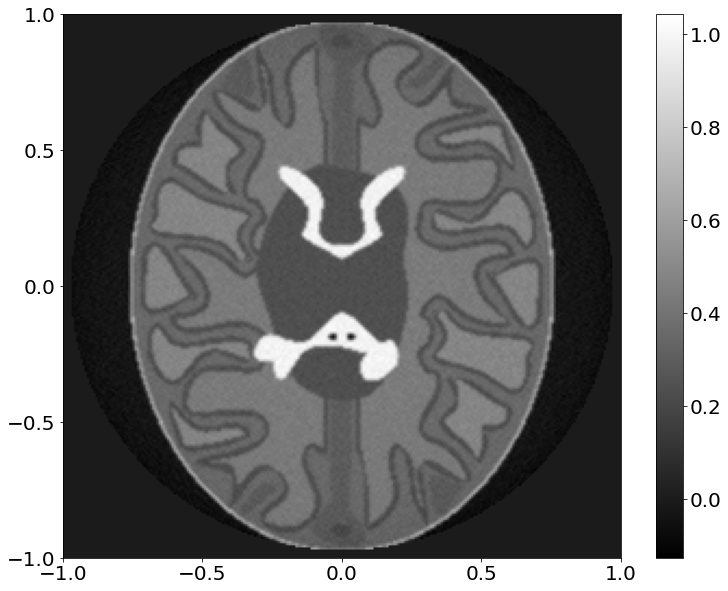}
	\end{minipage}
	\begin{minipage}{0.5\textwidth}
		\includegraphics[scale=0.25]{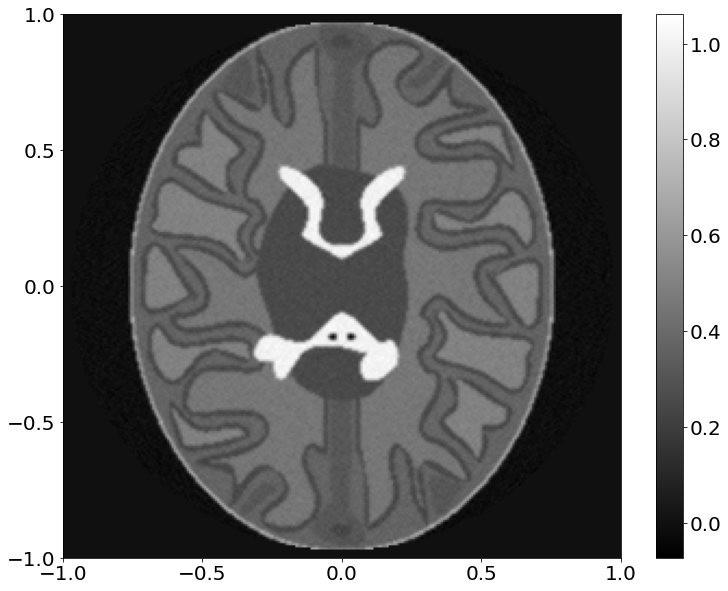}
	\end{minipage}
	\caption{Reconstructions with $20\%$ Gaussian noise added data. Top, left: $\mathbf{F}_{\infty}(\mathbf{n}_T)$. Top, right: $\mathbf{F}_T(\mathbf{n}_T)$. Below, left: $\mathbf{G}_{\infty}(\mathbf{d}_T)$. Below, right: $\mathbf{G}_{T}(\mathbf{d}_T)$.}
	\label{fig:recnoisy20}
\end{figure}
Despite the added noise, we conclude that the disturbed data didn't strongly affect the recovery of our head phantom. The artefacts outside of the phantom in the left reconstructions have become a bit stronger, whereas the right phantoms almost show no difference to the reconstructions with exact data.

In addition, we added more noise to the exact data and computed the numerical approximations of our inversion formulas with $40\%$ Gaussian noise. The approximated phantoms are shown in Figure \ref{fig:recnoisy40}.
\begin{figure}[h!]
	\begin{minipage}{0.5\textwidth}
		\includegraphics[scale=0.25]{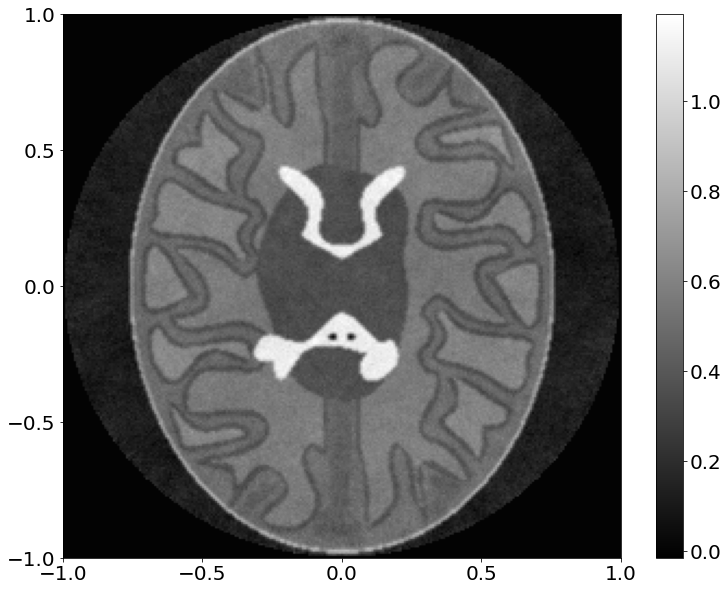}
	\end{minipage}
	\begin{minipage}{0.5\textwidth}
		\includegraphics[scale=0.25]{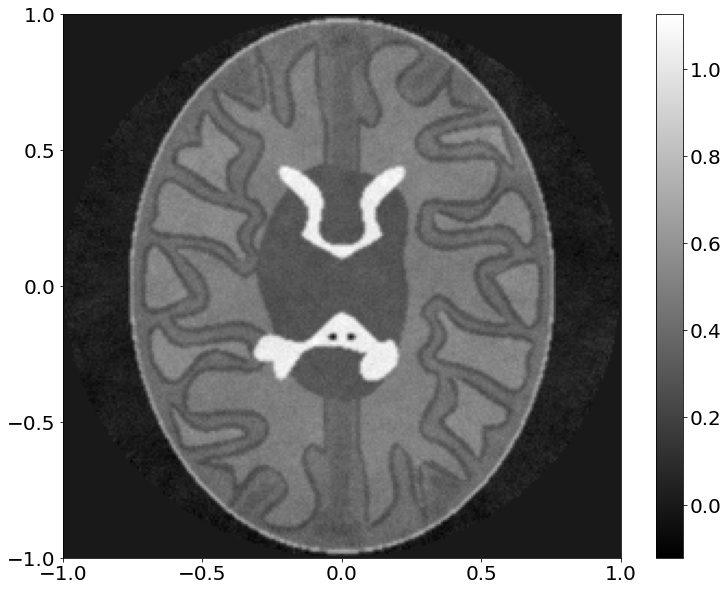}
	\end{minipage}
	\begin{minipage}{0.5\textwidth}
		\includegraphics[scale=0.25]{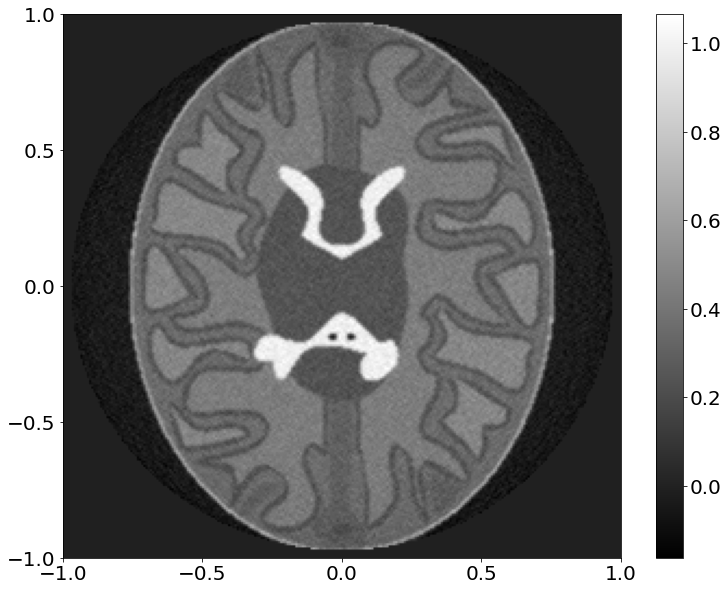}
	\end{minipage}
	\begin{minipage}{0.5\textwidth}
		\includegraphics[scale=0.25]{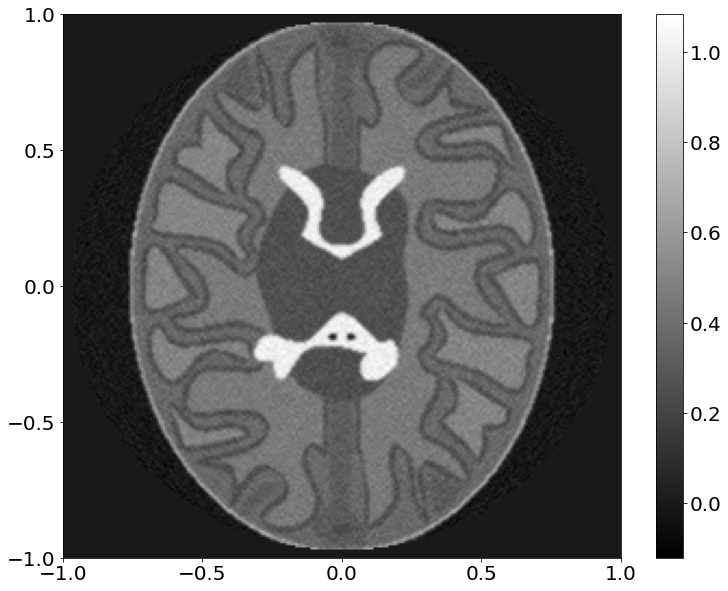}
	\end{minipage}
	\caption{Reconstructions with $40\%$ Gaussian noise added data. Top, left: $\mathbf{F}_{\infty}(\mathbf{n}_T)$. Top, right: $\mathbf{F}_T(\mathbf{n}_T)$. Below, left: $\mathbf{G}_{\infty}(\mathbf{d}_T)$. Below, right: $\mathbf{G}_{T}(\mathbf{d}_T)$.}
	\label{fig:recnoisy40}
\end{figure}
Because of the higher noise rate in the data, we now obtain discrete values within the interval $[-0.0157,0.4821]$ and $[-0.1225,0.373]$ outside of the phantom for $\mathbf{F}_{\infty}(\mathbf{n}_T)$ and $\mathbf{F}_T(\mathbf{n}_T)$, respectively. The approximated formulas for Dirichlet traces yield the ranges $[-0.1626,0.3276]$ in $\mathbf{G}_{\infty}(\mathbf{d}_T)$ and $[-0.1207,0.3487]$ in $\mathbf{G}_{T}(\mathbf{d}_T)$. Despite similar ranges, the discrete approximations of the exact formulas over finite time intervals show better results.

A clear difference can be detected between the two rows in Figure \ref{fig:rangecond} showing the numerical implementation of the two range conditions on the right-hand side in \eqref{eq:rangecondeven} and \eqref{eq:rangecondbteven}, which both should be theoretically to zero. The left reconstruction in the second row, which corresponds to $\mathbf{F}_{T}(\mathbf{d}_T)$ and equals the discretization of range condition \eqref{eq:rangecondbteven}, is nearly seven times lower than the error in the left reconstruction in the first row, which is the numerical implementation of range condition \eqref{eq:rangecondeven} for unbounded time intervals and is computed by $\mathbf{F}_{\infty}(\mathbf{d}_T)$.
The middle and right images in Figure \ref{fig:rangecond} illustrate the same numerical reconstructions as in the left column with $20\%$ and $40\%$ noisy data.
\begin{figure}
	\begin{minipage}{0.325\textwidth}
		\includegraphics[scale=0.1975]{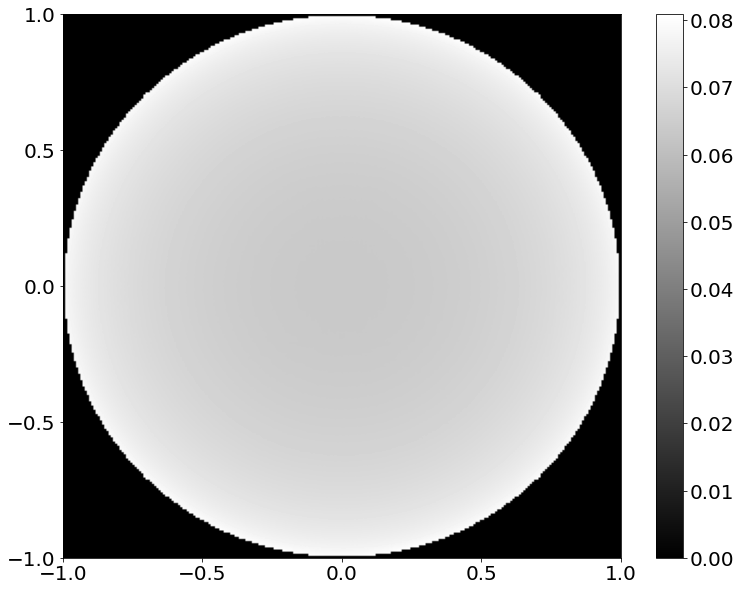}
	\end{minipage}
	\begin{minipage}{0.325\textwidth}
		\includegraphics[scale=0.1975]{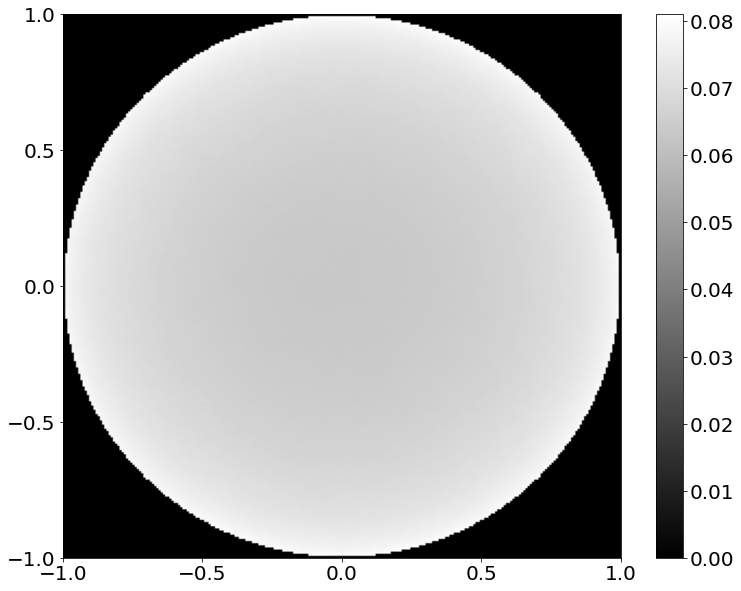}
	\end{minipage}
	\begin{minipage}{0.325\textwidth}
		\includegraphics[scale=0.1975]{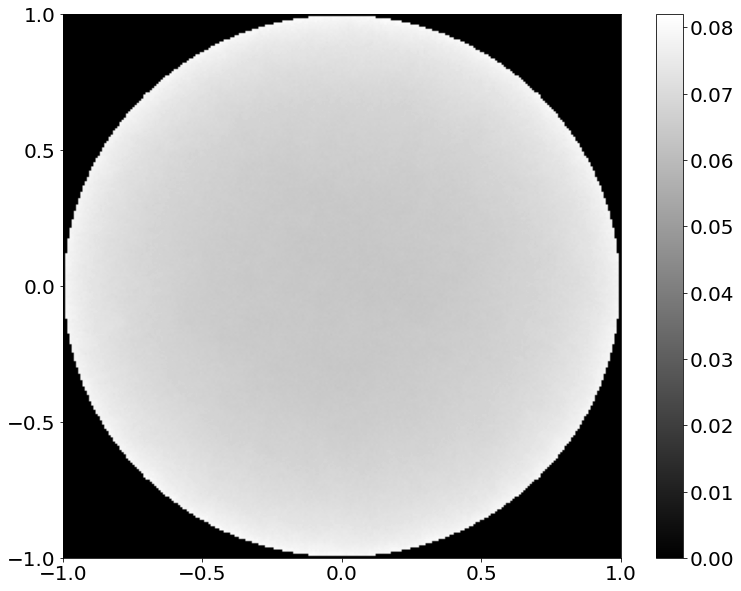}
	\end{minipage}\\
	\begin{minipage}{0.325\textwidth}
		\includegraphics[scale=0.1975]{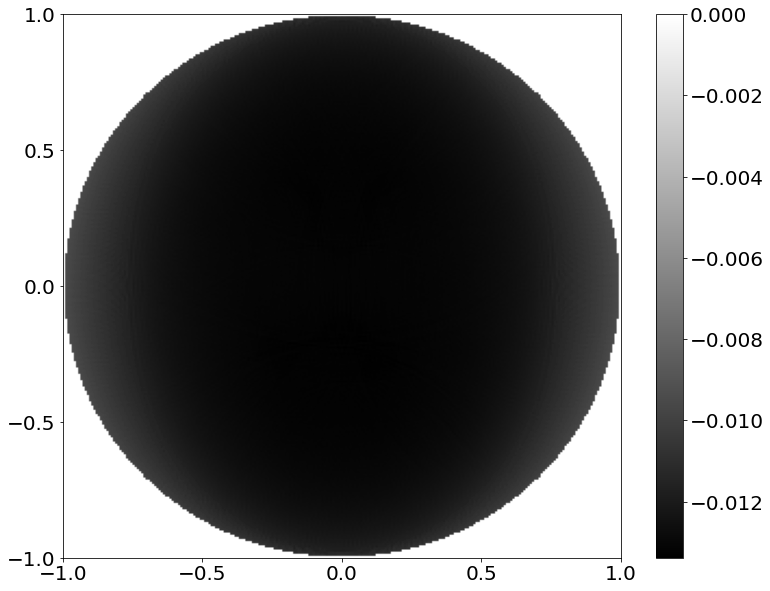}
	\end{minipage}
	\begin{minipage}{0.325\textwidth}
		\includegraphics[scale=0.1975]{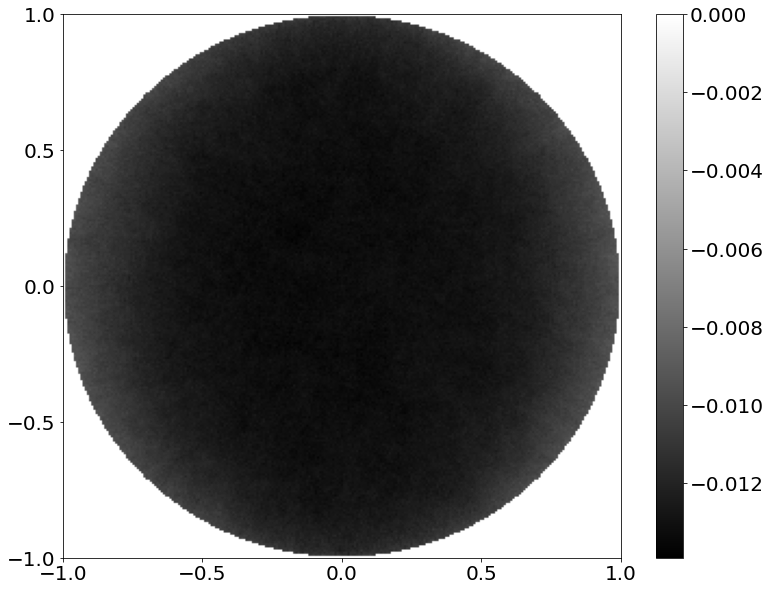}
	\end{minipage}
	\begin{minipage}{0.325\textwidth}
		\includegraphics[scale=0.1975]{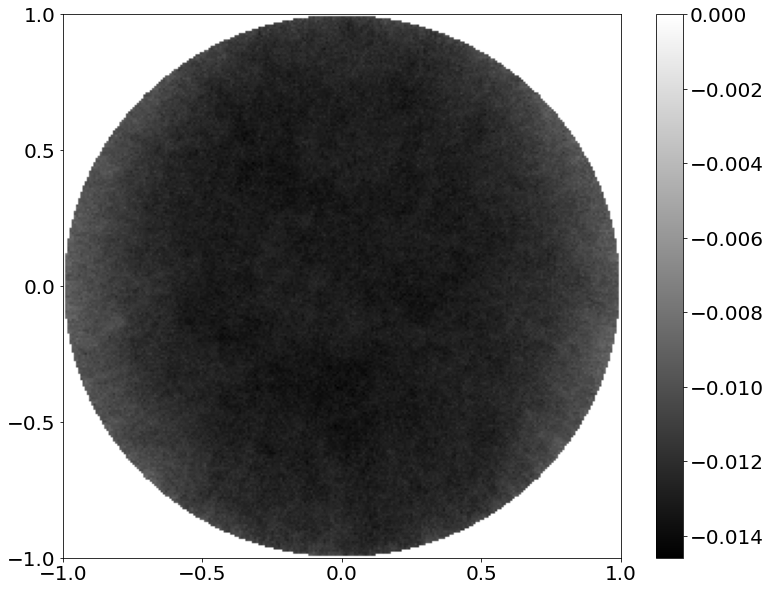}
	\end{minipage}
	\caption{Illustration of the range conditions \eqref{eq:rangecondeven} and \eqref{eq:rangecondbteven}, which should be theoretically equal to zero. Top (from left to right): $\mathbf{F}_{\infty}(\mathbf{d}_T)$ with exact data, $20\%$ Gaussian noise and $40\%$ Gaussian noise added data. Below (from left to right): $\mathbf{F}_T(\mathbf{d}_T)$ with exact, $20\%$ Gaussian noise and $40\%$ Gaussian noise added data.}
	\label{fig:rangecond}
\end{figure}

Lastly, we present the numerical reconstructions of the head phantom obtained by the new formulas \eqref{eq:invformmixedeven} and \eqref{eq:invformmixedbteven} for mixed data. As demonstrated in Figure \ref{fig:recmix}, the difference between the first row, which shows the reconstructions $\mathbf{F}_{\infty}(\mathbf{mix}_T)$ with our previous inversion formula for unbounded time intervals, and the second row, which illustrates the discretizations of $\mathbf{F}_{T}(\mathbf{mix}_T)$ with our new inversion formula for finite time intervals, is much more obvious than the reconstructions $\mathbf{F}_{\infty}(\mathbf{n}_T)$ and $\mathbf{F}_T(\mathbf{n}_T)$ in Figures \ref{fig:recexact}, \ref{fig:recnoisy20} and \ref{fig:recnoisy40}. This can explained as follows: The discrete formulas for mixed data $\mathbf{F}_{\infty}(\mathbf{mix}_T)$ and  $\mathbf{F}_T(\mathbf{mix}_T)$ can be expressed as the sum of the two discrete formulas $\mathbf{F}_{\infty}(\mathbf{d}_T)$ and $\mathbf{F}_{\infty}(\mathbf{n}_T)$, and $\mathbf{F}_{T}(\mathbf{d}_T)$ and $\mathbf{F}_{T}(\mathbf{n}_T)$, respectively. Because of the additional term $\mathbf{F}_{T}(\mathbf{d}_T)$, which is nearly seven times smaller than $\mathbf{F}_{\infty}(\mathbf{d}_T)$ as we already have seen in Figure \ref{fig:rangecond}, the error gets beside the use of $\mathbf{F}_{T}(\mathbf{n}_T)$ additionally smaller. This observation also holds for disturbed data (see Figure \ref{fig:rangecond}, middle, right).
\begin{figure}
	\begin{minipage}{0.325\textwidth}
		\includegraphics[scale=0.2]{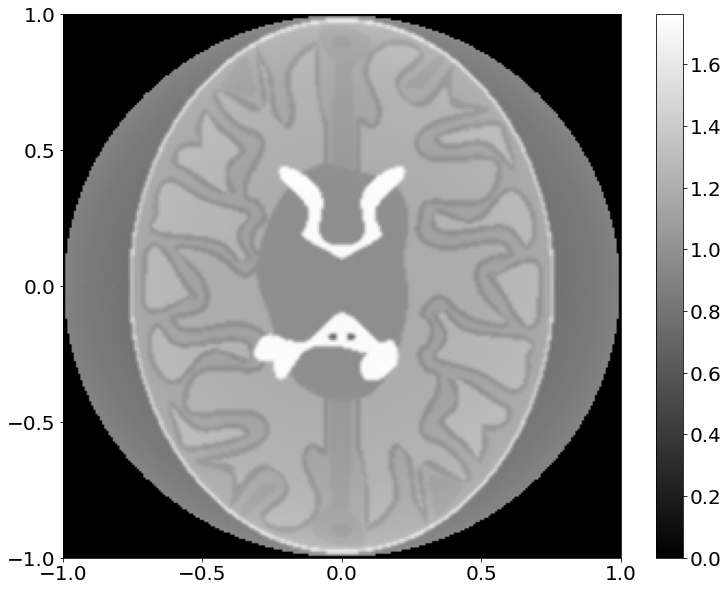}
	\end{minipage}
	\begin{minipage}{0.325\textwidth}
		\includegraphics[scale=0.2]{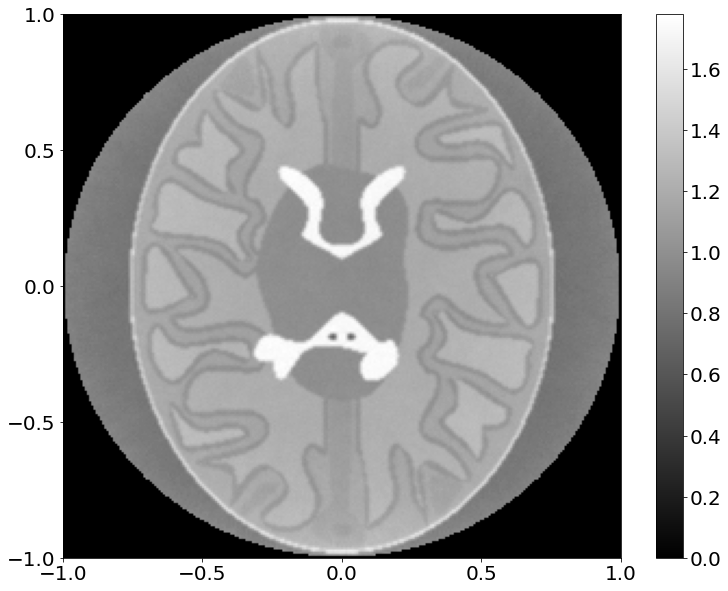}
	\end{minipage}
	\begin{minipage}{0.325\textwidth}
		\includegraphics[scale=0.2]{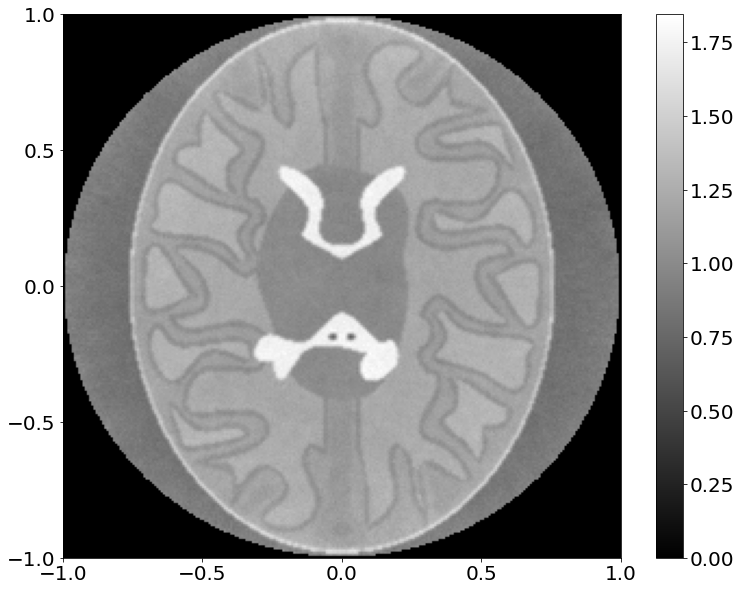}
	\end{minipage}\\
	\begin{minipage}{0.325\textwidth}
		\includegraphics[scale=0.2]{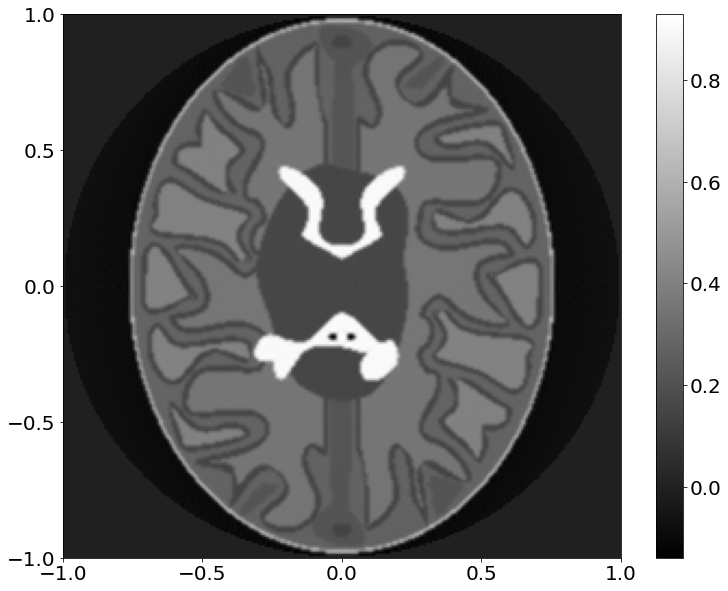}
	\end{minipage}
	\begin{minipage}{0.325\textwidth}
		\includegraphics[scale=0.2]{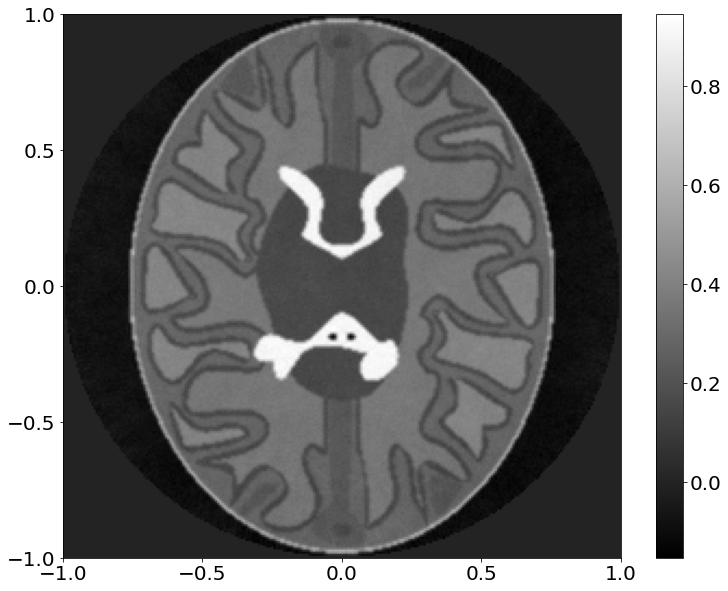}
	\end{minipage}
	\begin{minipage}{0.325\textwidth}
		\includegraphics[scale=0.2]{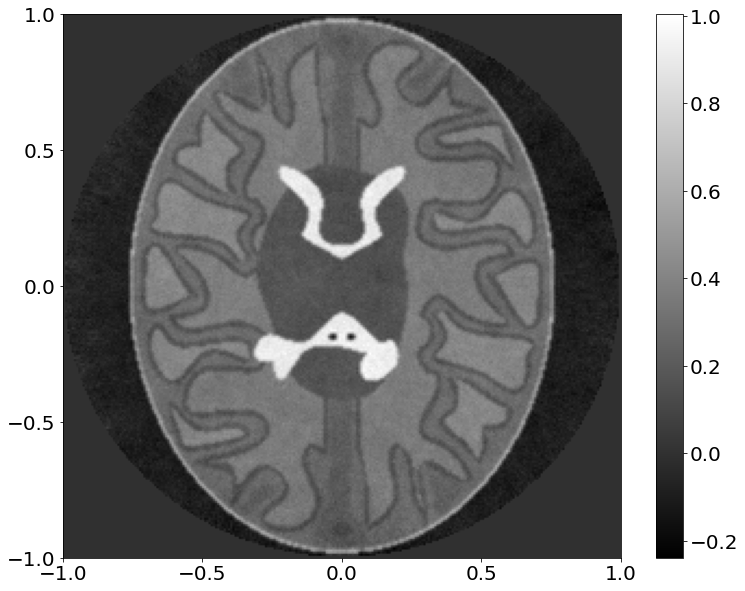}
	\end{minipage}
	\caption{Reconstructions with mixed data. Top (from left to right): $\mathbf{F}_{\infty}(\mathbf{mix}_T)$ with exact data, $20\%$ Gaussian noise and $40\%$ Gaussian noise added data. Below (from left to right): $\mathbf{F}_{T}(\mathbf{mix}_T)$ with exact, $20\%$ Gaussian noise and $40\%$ Gaussian noise added data.}
	\label{fig:recmix}
\end{figure}

The corresponding discrete $L^2$-errors of all reconstructions have also been evaluated and are summarized in Table \ref{tab:l2errT2}. The discrete $L^2$-error of $\mathbf{F}_{\infty}(\mathbf{n}_T)$, for example, is defined as \[\left(\sum_{(i,j)\in\Omega_{\Delta x}}\abs{\mathbf{F}_{\infty}(\mathbf{n}_T)[i,j]-\mathbf{F}[i,j]}^2{\Delta x}^2\right)^{1/2},\] where $\Omega_{\Delta x}\coloneqq\set{(i,j)\in\set{0,\ldots,N-1}^2\mid x_{i,j}\in\B^2}$. The errors for the other reconstructions are defined in an analogous way.
\begin{table}
	\center
	{\fontsize{10}{11}
	\begin{tabular}{l || c c | c c | c c | c c}
		$\text{noise}$ & $\mathbf{F}_{\infty}(\mathbf{n}_T)$ & $\mathbf{F}_{T}(\mathbf{n}_T)$ & $\mathbf{G}_{\infty}(\mathbf{d}_T)$ & $\mathbf{G}_{T}(\mathbf{d}_T)$ & $\mathbf{F}_{\infty}(\mathbf{d}_T)$ & $\mathbf{F}_{T}(\mathbf{d}_T)$ & $\mathbf{F}_{\infty}(\mathbf{mix}_T)$ & $\mathbf{F}_{T}(\mathbf{mix}_T)$\\\hline
		$\text{exact}$ & $0.19439$ & $0.09607$ & $0.09622$ & $0.08303$ & $0.11905$ & $0.02143$ & $1.36884$ & $0.18521$\\
$20\%$ & $0.20181$ & $0.09962$ & $0.09927$ & $0.08664$ & $0.11915$ & $0.02139$ & $1.35957$ & $0.19395$\\
$40\%$ & $0.22693$ & $0.11669$ & $0.10759$ & $0.09607$ & $0.11914$ & $0.02141$ & $1.39782$ & $0.18419$
	\end{tabular}}
	\caption{Discrete $L^2$-error of all reconstructions for $T=2$.}
	\label{tab:l2errT2}
\end{table}
\subsection{Varying the end time $T$}
In the last section of the article, we conclude our numerical results by calculating the single reconstructions for various end times $T$ and compute their discrete $L^2$-error. In Figure \ref{fig:l2error}, the corresponding error plots of the single numerical reconstructions are shown. Note that the diagrams show the error of the reconstructions with exact data, $20\%$ Gaussian noise and $40\%$ Gaussian noise added data. Moreover, due to the higher computation time for the simulation of the wave data, we selected a bigger time difference of $\Delta t=10^{-3}$.\\
\begin{figure}[h!]
\begin{minipage}{0.5\textwidth}
\begin{center}
	\begin{tikzpicture}
		\def\sc{1.25}
		\begin{axis}[xlabel=end time $T$,ylabel=$L^2$-error,width=6.32cm,height=5.175cm]
			\addplot[thick,lightgray,mark=*,mark size=1.4pt] table[x index=0, y index=1, row sep=\\]{
				2 0.2011418578081984 \\ 4 0.08540689209683536 \\ 6 0.08554012974914825 \\ 8 0.08564860436516977 \\
			};
			\addplot[thick,lightgray,dashed,mark=square*,mark size=1.4pt] table[x index=0, y index=1, row sep=\\]{
				2 0.09720742292937443 \\ 4 0.08547232363176221 \\ 6 0.08558249276787355 \\ 8 0.08567743726762787 \\
			};
			\addplot[thick,gray,mark=*,mark size=1.4pt] table[x index=0, y index=1, row sep=\\]{
				2 0.19255901770436798 \\ 4 0.10915027817947273 \\ 6 0.11688311576594197 \\ 8 0.11369491665950579 \\
			};
			\addplot[thick,gray,dashed,mark=square*,mark size=1.4pt] table[x index=0, y index=1, row sep=\\]{
				2 0.11582724416159765 \\ 4 0.10881648788650267 \\ 6 0.11449395760174262 \\ 8 0.11351529216973727 \\
			};
			\addplot[thick,black,mark=*,mark size=1.4pt] table[x index=0, y index=1, row sep=\\]{
				2 0.2733802070617781 \\ 4 0.155277277064195 \\ 6 0.19109212284476787 \\ 8 0.167856389044726 \\
			};
			\addplot[thick,black,dashed,mark=square*,mark size=1.4pt] table[x index=0, y index=1, row sep=\\]{
				2 0.18127665035397258 \\ 4 0.15394820829157463 \\ 6 0.17687026269163977 \\ 8 0.16555744131315422 \\
			};
		\end{axis}
	\end{tikzpicture}
\end{center}
\end{minipage}
\begin{minipage}{0.5\textwidth}
\begin{center}
	\begin{tikzpicture}
		\begin{axis}[xlabel=end time $T$,ylabel=$L^2$-error,width=6.32cm,height=5.175cm]
			\addplot[thick,lightgray,mark=*,mark size=1.4pt] table[x index=0, y index=1, row sep=\\]{
				2 1.408766344909571 \\ 4 0.2018103784134135 \\ 6 0.11359520206329808 \\ 8 0.09473030035252065 \\
			};
			\addplot[thick,lightgray,dashed,mark=square*,mark size=1.4pt] table[x index=0, y index=1, row sep=\\]{
				2 0.19424664354315072 \\ 4 0.1252386514855261 \\ 6 0.09740194486734649 \\ 8 0.09077604714276982 \\
			};
			\addplot[thick,gray,mark=*,mark size=1.4pt] table[x index=0, y index=1, row sep=\\]{
				2 1.4304682024932973 \\ 4 0.22578852381933365 \\ 6 0.12782611408030134 \\ 8 0.10917611279137554 \\
			};
			\addplot[thick,gray,dashed,mark=square*,mark size=1.4pt] table[x index=0, y index=1, row sep=\\]{
				2 0.20100867042583817 \\ 4 0.1309960634302043 \\ 6 0.12168548058388709 \\ 8 0.13327688659828074 \\
			};
			\addplot[thick,black,mark=*,mark size=1.4pt] table[x index=0, y index=1, row sep=\\]{
				2 1.3074953891040602 \\ 4 0.1698332312594166 \\ 6 0.17897002985051716 \\ 8 0.16854126201393665 \\
			};
			\addplot[thick,black,dashed,mark=square*,mark size=1.4pt] table[x index=0, y index=1, row sep=\\]{
				2 0.27402675527686016 \\ 4 0.23233696010617783 \\ 6 0.15542273422551692 \\ 8 0.1892838523718022 \\
			};
		\end{axis}
	\end{tikzpicture}
\end{center}
\end{minipage}\vspace{0.25cm}\\
\begin{minipage}{0.5\textwidth}
\begin{center}
	\begin{tikzpicture}
		\begin{axis}[xlabel=end time $T$,ylabel=$L^2$-error,width=6.32cm,height=5.175cm]
			\addplot[thick,lightgray,mark=*,mark size=1.4pt] table[x index=0, y index=1, row sep=\\]{
				2 0.11417818506889632 \\ 4 0.09843729484915373 \\ 6 0.09719241111756874 \\ 8 0.09598104400580848 \\
			};
			\addplot[thick,lightgray,dashed,mark=square*,mark size=1.4pt] table[x index=0, y index=1, row sep=\\]{
				2 0.10315287181115776 \\ 4 0.09851787541544679 \\ 6 0.09723266259177932 \\ 8 0.09600797369527098 \\
			};
			\addplot[thick,gray,mark=*,mark size=1.4pt] table[x index=0, y index=1, row sep=\\]{
				2 0.1301637227135487 \\ 4 0.11369081016655977 \\ 6 0.11285383304831437 \\ 8 0.11161371564102547 \\
			};
			\addplot[thick,gray,dashed,mark=square*,mark size=1.4pt] table[x index=0, y index=1, row sep=\\]{
				2 0.1201789322913684 \\ 4 0.11368991160708838 \\ 6 0.11284803774964125 \\ 8 0.11160564617209791 \\
			};
			\addplot[thick,black,mark=*,mark size=1.4pt] table[x index=0, y index=1, row sep=\\]{
				2 0.1680521071246963 \\ 4 0.15050907246083386 \\ 6 0.15088450525479208 \\ 8 0.1477457120137633 \\
			};
			\addplot[thick,black,dashed,mark=square*,mark size=1.4pt] table[x index=0, y index=1, row sep=\\]{
				2 0.16014292561822316 \\ 4 0.15036831793239117 \\ 6 0.15078331786361987 \\ 8 0.14767031174383463 \\
			};
		\end{axis}
	\end{tikzpicture}
\end{center}
\end{minipage}
\begin{minipage}{0.5\textwidth}
\begin{center}
	\begin{tikzpicture}
		\begin{axis}[xlabel=end time $T$,ylabel=$L^2$-error,width=6.32cm,height=5.175cm,yticklabels={$0$,$0.05$,$0.1$},ytick={0,0.05,0.1}]
			\addplot[thick,lightgray,mark=*,mark size=1.4pt] table[x index=0, y index=1, row sep=\\]{
				2 0.12242445946888533 \\ 4 0.018033230514501567 \\ 6 0.007339277795984941 \\ 8 0.0038671257301388425 \\
			};
			\addplot[thick,lightgray,dashed,mark=square*,mark size=1.4pt] table[x index=0, y index=1, row sep=\\]{
				2 0.02194985509899863 \\ 4 0.008842951066622739 \\ 6 0.00440852538320653 \\ 8 0.002733170581949765 \\
			};
			\addplot[thick,gray,mark=*,mark size=1.4pt] table[x index=0, y index=1, row sep=\\]{
				2 0.12245618593346226 \\ 4 0.018288540548960345 \\ 6 0.008078156078847407 \\ 8 0.003748412766508454 \\
			};
			\addplot[thick,gray,dashed,mark=square*,mark size=1.4pt] table[x index=0, y index=1, row sep=\\]{
				2 0.02178182835156803 \\ 4 0.008729068554548055 \\ 6 0.0040221111396436 \\ 8 0.003066094744531736 \\
			};
			\addplot[thick,black,mark=*,mark size=1.4pt] table[x index=0, y index=1, row sep=\\]{
				2 0.12138934412581316 \\ 4 0.017873871014754208 \\ 6 0.00888340265428734 \\ 8 0.003752647897704387 \\
			};
			\addplot[thick,black,dashed,mark=square*,mark size=1.4pt] table[x index=0, y index=1, row sep=\\]{
				2 0.02253215346310185 \\ 4 0.009140644179453783 \\ 6 0.0036978326050199147 \\ 8 0.0036852047966249865 \\
			};
		\end{axis}
	\end{tikzpicture}
\end{center}
\end{minipage}
\caption{$L^2$-Error of the single reconstructions with  exact data (light gray lines), $20\%$ Gaussian noise (gray lines) and $40\%$ Gaussian noise (black lines) for $T=2,4,6,8$: Top, left: $\mathbf{F}_{\infty}(\mathbf{n}_T)$ (solid) and $\mathbf{F}_T(\mathbf{n}_T)$ (dashed). Top, right: $\mathbf{F}_{\infty}(\mathbf{mix}_T)$ (solid) and $\mathbf{F}_{T}(\mathbf{mix}_T)$ (dashed). Below, left: $\mathbf{G}_{\infty}(\mathbf{d}_T)$ (solid) and $\mathbf{G}_{T}(\mathbf{d}_T)$ (dashed). Below, right: $\mathbf{F}_{\infty}(\mathbf{d}_T)$ (solid) and $\mathbf{F}_T(\mathbf{d}_T)$ (dashed).}
\label{fig:l2error} 
\end{figure}
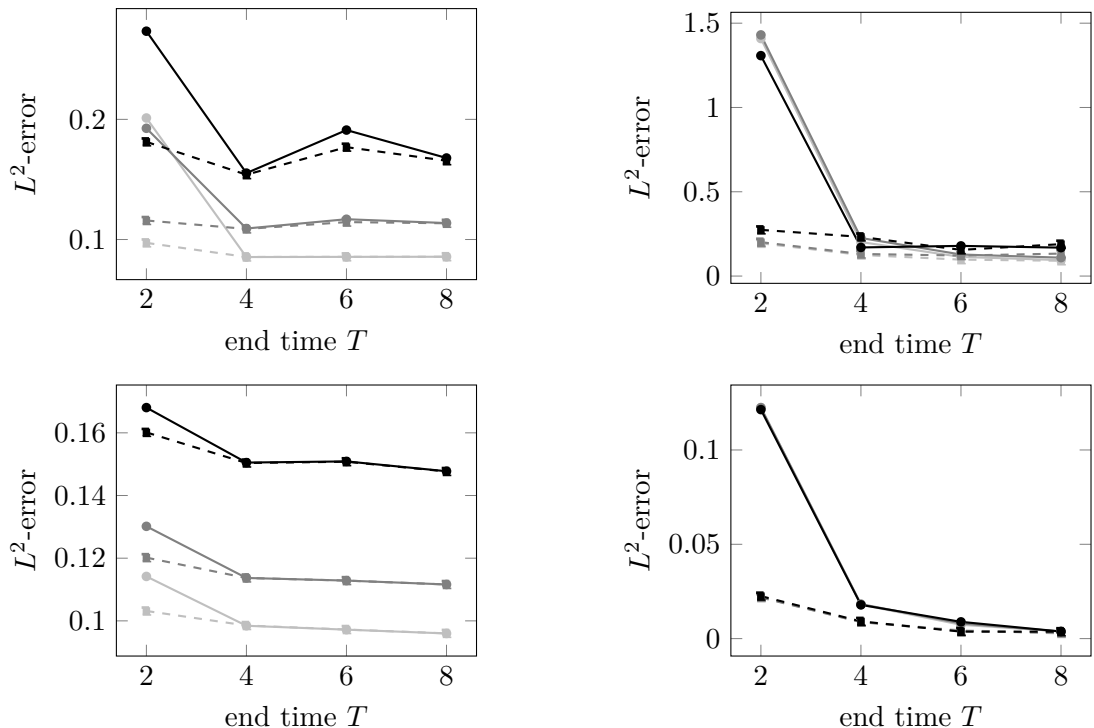
We can clearly see that biggest difference between the reconstructions are for the shortest end time $T=2$. With very few exceptions, the four discretized versions of the inversion formulas using finite time intervals have a smaller $L^2$-error for all end times as the ones using infinite time measurements. Moreover, we can see that the error values are tending to decrease as $T$ is getting larger. We also observe that the larger the end time $T$ is, the smaller the difference between the solid and dashed lines gets. Nevertheless, since in practice the acoustic waves are measured as shortly as possible due to external influences, the use of the new formulas offers a clear advantage for recovering the initial data of the wave equation.
\section{Conclusion}
\label{sec:conclusion}
In this article, we presented new reconstruction formulas for the inverse source problem in PAT requiring only measurements on finite time intervals.
To the best of our knowledge, these are the first reconstruction formulas which utilize photoacoustic measurements in the form of Dirichlet, Neumann and mixed traces on finite time intervals $(0,T)$. The formulas for the Dirichlet and Neumann data yield exact reconstruction for elliptical domains and the formula for mixed data exact reconstruction for spheres in $\R^n$. Similar to existing explicit wave inversion formulas in \cite{FinHalRak07,FinPatRak04,Kun07,Ngu09,Hal14} using data on $(0,\infty)$, they are of the filtered backprojection type and serve as basis of efficient and stable reconstruction algorithms. However, the differential operator $(\partial_t t^{-1})^{(n-2)/2}$ for even dimensions and $(\partial_t t^{-1})^{(n-3)/2}$ for odd dimensions inside integrals leads to an amplification of noise, which can be reduced by regularization similar to \cite{Lou96}.

In the even-dimensional case, we found a kernel function $k_T: (0,T)^2\to\R$ that can be used to recover the absorption coefficient independently of the spatial dimension $n$ from all three traces. We also observe that the formulas \eqref{eq:invneumanneven}, \eqref{eq:invdirichleteven} and \eqref{eq:invformmixedeven} for unbounded time intervals can also be written in the form of a kernel function defined by
\begin{equation*}
	k_{\infty}\colon (0,\infty)^2\to\R\colon (r_1,r_2)\mapsto\begin{cases}
		\frac{1}{\sqrt{r_2^2-r_1^2}},\quad &r_2>r_1,\\
		0,\quad &\text{else}.
	\end{cases}
\end{equation*}
Then, it is not hard to see that $(k_T)_{T>0}$ converges pointwise by a zero extension on $(0,\infty)^2$ to $k_{\infty}$ for $T\to\infty$. Note that the formulas for mixed traces \eqref{eq:invformmixedeven}, \eqref{eq:invformmixedbteven} and \eqref{eq:invformmixedbtodd}, except the weight $b$ of the Neumann trace, which appears as a factor in front of integral, are independent of the weighting factors $a,b$. This follows from the range conditions \eqref{eq:rangecondeven}, \eqref{eq:rangecondbteven} and \eqref{eq:rangecondbtodd}, since the application of the formula for Neumann traces on Dirichlet traces equals zero. It could be of great interest to derive specific inversion formulas, which use filters that are adjusted to the weights $(a,b)$ and might offer improved robustness.

In the simulations, we have seen that the numerical reconstructions of our new inversion formulas clearly show better results than the old formulas for unlimited time measurements, in particular the formula for mixed data. The relevance of our new inversion formulas for real-world data is that only data up to finite time $T$ are required. Other formulas using $T=\infty$ suggest that $T$ should be taken as large as possible in the actual measurement. This, however, results in an increased amount of noise and fake signals in the data, which are avoided by finite time inversion formulas. Therefore, the new photoacoustic inversion formulas provide a significant improvement to existing inversion formulas for unbounded time intervals, and thus for real-world applications in PAT.

\section{Appendix}
\label{sec:appendix}
\subsection{Solution formulas for the wave equation}
\label{subsec:solwaveeq}
As can be observed in \cite{DreHal20,DreHal21}, the derivation of the inversion formulas for Neumann traces are largely based on the analytic expression of the solution of the standard wave equation \eqref{eq:waveeq}. For example, in \cite{Eva10}, there is presented the well-known solution formula
\begin{equation}
	\label{eq:solwaveeqeven}
	\begin{aligned}
	u(x,t)=\frac{1}{\gamma_n\omega_n}&\Bigg[\partial_t\left(\frac{1}{t}\partial_t\right)^{\frac{n-2}{2}}\left(\int_{\B_t^n(x)}\frac{f(y)}{\sqrt{t^2-\norm{y-x}^2}}\d{y}\right)\\
	&+\left(\frac{1}{t}\partial_t\right)^{\frac{n-2}{2}}\left(\int_{\B_t^n(x)}\frac{g(y)}{\sqrt{t^2-\norm{y-x}^2}}\d{y}\right)\Bigg].
	\end{aligned}
\end{equation}
for even dimensions $n\geq 2$ and $f,g\in C_c^\infty(\Omega)$, where $\B_t^n(x)$ is the open ball with radius $t$ and center $x$ in $\R^n$ and $\gamma_n\coloneqq 2\cdot 4\cdots (n-2)\cdot n$. Another representation formula is given by
\begin{equation}
	\label{eq:solwaveeqeven2}
	\begin{aligned}
		u(x,t)=\frac{n}{\gamma_n}&\Bigg[\partial_t\left(\int_0^t\frac{r}{\sqrt{t^2-r^2}}\left(\frac{1}{r}\partial_r\right)^{\frac{n-2}{2}}\left(r^{n-2}\M f(x,r)\right)\d{r}\right)\\
		&+\left(\int_0^t\frac{r}{\sqrt{t^2-r^2}}\left(\frac{1}{r}\partial_r\right)^{\frac{n-2}{2}}\left(r^{n-2}\M g(x,r)\right)\d{r}\right)\Bigg],
	\end{aligned}
\end{equation}
(see, for example, \cite[Lemma 2.2]{DreHal21}). Based on \eqref{eq:solwaveeqeven2}, we now present an analytic expression for the directional derivative of the solution of the wave equation along a vector $v\in\R^n$, being used in section \ref{sec:timeboundinveven}.
Before that, we present another technical result for the operator $\left(\frac{1}{r}\partial_r\right)^{k}r^{n-2}\circ \M$, appearing in formula \eqref{eq:solwaveeqeven2}. Formula \eqref{eq:solwaveeqeven} will be used in section \ref{sec:timeboundinvevenmixed}.
\begin{lemma}
\label{lem:relsopspmean}
Let $n\geq 2$ be an integer and  $f\in C_c^\infty(\Omega)$. For every $k\in\N$ and $(x,r)\in\R^n\times(0,\infty)$, we have
	\begin{equation*}
		\left(\frac{1}{r}\partial_r\right)^{k}r^{n-2}\M f(x,r)	=\sum_{l=0}^k c_{k,l}^{(n)}r^{n-(2k+1)+l}\frac{1}{\sigma(\Sp^{n-1})}\int_{\Sp^{n-1}}\sum_{i\in\set{1,\ldots,n}^l}\partial^if(x+ry)y^i\d{\sigma(y)},
	\end{equation*}
	where $\partial^i\coloneqq \partial_{i_1}\ldots\partial_{i_k}$, $y^i\coloneqq y_{i_1}\cdot\ldots\cdot y_{i_k}$ and the coefficients are recursively defined by $c_{0,0}^{(n)}\coloneqq 1$, $c_{1,0}^{(n)}\coloneqq n-2$, $c_{1,1}^{(n)}\coloneqq 1$, $c_{\tilde{k},0}^{(n)}\coloneqq c_{\tilde{k}-1,0}^{(n)}(n-2\tilde{k})$, $c_{\tilde{k},\tilde{k}}^{(n)}\coloneqq 1$ and $c_{\tilde{k},l}^{(n)}\coloneqq c_{\tilde{k}-1,l-1}^{(n)}+c_{\tilde{k}-1,l}^{(n)}(n-2\tilde{k}+l)$ for all $\tilde{k}\in\set{2,\ldots,k}$ and $l\in\{1,\ldots,\tilde{k}-1\}$.
\end{lemma}
\begin{proof}
	For the proof of this identity, we refer to \cite[Lemma 3.4]{DreHal21}, where a proof of a similar identity can be found.
\end{proof}
\begin{prop}
	\label{prop:dvsolwaveeqeven}
	Let $n\geq 2$ be even, $v$ be a vector $\R^n$, $\Omega\subset\R^n$ an open domain and $f,g\in C_c^\infty(\Omega)$. Moreover, let $u\colon\R^n\times [0,\infty)\to\R$ be the solution of \eqref{eq:waveeq}.
	\begin{enumerate}
		\item For all $(x,t)\in\R^n\times(0,\infty)$ we have
		\begin{equation}
			\label{eq:dvsolwaveeqeven1}
			\begin{aligned}
				D_v u(x,t)=\frac{n}{\gamma_n}&\Bigg[\partial_t\left(\int_0^t\frac{r}{\sqrt{t^2-r^2}}\left(r^{-1}\partial_r\right)^{\frac{n-2}{2}}\left(r^{n-2}D_v\M f(x,r)\right)\d{r}\right)\\
				&+\left(\int_0^t\frac{r}{\sqrt{t^2-r^2}}\left(r^{-1}\partial_r\right)^{\frac{n-2}{2}}\left(r^{n-2}D_v\M g(x,r)\right)\d{r}\right)\Bigg].
			\end{aligned}
		\end{equation}
		\item For $x\in\partial\Omega$ and $k\in\N$ it holds
		\begin{equation}
			\label{eq:dvsolwaveeqeven2}
			\begin{aligned}
				\left(\partial_t t^{-1}\right)^k & D_v u(x,t)\\
=&\frac{n}{\gamma_n}\Bigg[\left(\int_0^t\frac{r}{t\sqrt{t^2-r^2}}\left(\partial_r r \left(r^{-1}\partial_r\right)^{\frac{n-2}{2}+k}r^{n-2}D_v\M f(x,r)\right)\d{r}\right)\\
				&+\left(\partial_t t^{-1}\right)^k\left(\int_0^t\frac{r}{\sqrt{t^2-r^2}}\left(r^{-1}\partial_r\right)^{\frac{n-2}{2}}\left(r^{n-2}D_v\M g(x,r)\right)\d{r}\right)\Bigg].
			\end{aligned}
		\end{equation}
	\end{enumerate}
\end{prop}
\begin{proof}
	\begin{enumerate}[wide=\parindent,label=(\roman*)]
		\item We start with the proof of the first identity. From Lemma \ref{lem:relsopspmean}, we see that $\partial_i\left(\frac{1}{r}\partial_r\right)^{\frac{n-2}{2}}r^{n-2}\M f$ is a bounded function for $1\leq i\leq n$. Therefore, interchanging the time derivative with the differential operator $\partial_i$, differentiating under the integral sign and interchanging $\partial_i$ with $\left(\frac{1}{r}\partial_r\right)^{\frac{n-2}{2}}r^{n-2}\M f$ yield
	\begin{equation*}
		\begin{aligned}
			\partial_i u(x,t)=\frac{n}{\gamma_n}&\Bigg[\partial_t\left(\int_0^t\frac{r}{\sqrt{t^2-r^2}}\left(\frac{1}{r}\partial_r\right)^{\frac{n-2}{2}}\left(r^{n-2}\partial_i\M f(x,r)\right)\d{r}\right)\\
			&+\left(\int_0^t\frac{r}{\sqrt{t^2-r^2}}\left(\frac{1}{r}\partial_r\right)^{\frac{n-2}{2}}\left(r^{n-2}\partial_i\M g(x,r)\right)\d{r}\right)\Bigg]
		\end{aligned}
	\end{equation*}
	for $1\leq i\leq n$. Next, we apply the above relation and \eqref{eq:reldirder} on $D_v u(x,t)$ to deduce
	\begin{align*}
		D_v u(x,t)=&\frac{n}{\gamma_n}\sum_{i=1}^nv_i\Bigg[\partial_t\left(\int_0^t\frac{r}{\sqrt{t^2-r^2}}\left(\frac{1}{r}\partial_r\right)^{\frac{n-2}{2}}\left(r^{n-2}\partial_i\M f(x,r)\right)\d{r}\right)\\
			&+\left(\int_0^t\frac{r}{\sqrt{t^2-r^2}}\left(\frac{1}{r}\partial_r\right)^{\frac{n-2}{2}}\left(r^{n-2}\partial_i\M g(x,r)\right)\d{r}\right)\Bigg]\\
			=&\frac{n}{\gamma_n}\Bigg[\partial_t\left(\int_0^t\frac{r}{\sqrt{t^2-r^2}}\left(\frac{1}{r}\partial_r\right)^{\frac{n-2}{2}}\left(r^{n-2}\sum_{i=1}^nv_i\partial_i\M f(x,r)\right)\d{r}\right)\\
			&+\left(\int_0^t\frac{r}{\sqrt{t^2-r^2}}\left(\frac{1}{r}\partial_r\right)^{\frac{n-2}{2}}\left(r^{n-2}\sum_{i=1}^nv_i\partial_i\M g(x,r)\right)\d{r}\right)\Bigg].	
	\end{align*}
	Finally, using \eqref{eq:reldirder} inside the integrals again shows the desired identity.
		\item First, we consider the case $k=1$. Applying integration by parts on the first term on the right-hand side in \eqref{eq:dvsolwaveeqeven1} yields the sum of two boundary terms plus the integral \[\int_0^t\sqrt{t^2-r^2}\partial_r\left(r^{-1}\partial_r\right)^{\frac{n-2}{2}}r^{n-2}D_v\M f(x,r)\d{r}.\] Since $f$ has compact support in $\Omega$ and $y\in\partial\Omega$, we see from the definition of the spherical mean operator that $\left(r^{-1}\partial_r\right)^{\frac{n-2}{2}}r^{n-2}D_v\M f(x,r)=0$ for sufficient small $r>0$. Therefore, both boundary terms are equal to zero. Then, using Leibniz's integral rule implies \[\left(t^{-1}\partial_t\right)D_v u(x,t)=\int_0^t\frac{r}{\sqrt{t^2-r^2}}\left(r^{-1}\partial_r\right)^{\frac{n-2}{2}+1}r^{n-2}D_v\M f(x,r)\d{r}.\] Then, the second identity follows from the integral identity \[\partial_t\int_0^t \frac{rh(r)}{\sqrt{t^2-r^2}}\d{r}=\frac{1}{t}\int_0^t \frac{r\partial_r rh(r)}{\sqrt{t^2-r^2}}\d{r},\] given in \cite[Proposition 3.1]{FinHalRak07}. The case $k>1$ can be shown by repeating the first argument $k$-times and an application of the above integral identity in the last step.\qedhere
	\end{enumerate}
\end{proof}
\begin{remark}
	Similarly as in Proposition \ref{prop:dvsolwaveeqeven}, one can show that the formula
	\begin{equation}
		\label{eq:solwaveeqeven3}
		\begin{aligned}
			\left(\partial_t t^{-1}\right)^k  u(x,t)=&\frac{n}{\gamma_n}\Bigg[\left(\int_0^t\frac{r}{t\sqrt{t^2-r^2}}\left(\partial_r r \left(r^{-1}\partial_r\right)^{\frac{n-2}{2}+k}r^{n-2}\M f(x,r)\right)\d{r}\right)\\
			&+\left(\partial_t t^{-1}\right)^k\left(\int_0^t\frac{r}{\sqrt{t^2-r^2}}\left(r^{-1}\partial_r\right)^{\frac{n-2}{2}}\left(r^{n-2}\M g(x,r)\right)\d{r}\right)\Bigg].
		\end{aligned}
	\end{equation}
	for all $(x,t)\in\R^n\times(0,\infty)$ holds.
\end{remark}
\subsection{Abel integral equations}
\label{sec:abelinteq}
Typical Abel integral equations have the form
\begin{equation}
	\label{eq:abelinteq}
	\int_a^x \frac{u(r)}{(x-r)^{1-\alpha}}\d{r}=f(x),
\end{equation}
where $-\infty\leq a<b\leq\infty$, $a<x<b$ and $\alpha\in(0,1)$. Here, $f\colon(a,b)\to\R$ is a given function and $u\colon(a,b)\to\R$ the function to be determined. In \cite{GorVes91}, it is mentioned that, for example, if $f$ is absolute continuous, then \eqref{eq:abelinteq} has a unique solution in $L^1((a,b))$ and $u$ is given by the formula
\begin{equation}
	\label{eq:solabelinteq}
	u(r)=\frac{1}{\Gamma(1-\alpha)\Gamma(\alpha)}\frac{d}{dr}\int_a^r \frac{f(x)}{(r-x)^\alpha}\d{x},\quad r\in(a,b).
\end{equation}
Under the assumption $x\in\partial\Omega$, we can transform equations \eqref{eq:solwaveeqeven2} and \eqref{eq:dvsolwaveeqeven1} into the standard form of the Abel integral equation \eqref{eq:abelinteq} for $\alpha=1/2$, which together with \eqref{eq:solabelinteq} allows to derive the following result:
\begin{prop}
	\label{prop:solabelinteqwaveeqeven}
	Let $n\geq 2$ be even, $x\in\partial\Omega$, $v\in\R^n$ and $u\colon\R^n\times [0,\infty)\to\R$ the solution of \eqref{eq:waveeq} with initial data $(f,0)$. Then, we have for every radius $r>0$
	\begin{equation}
		\label{eq:solabelinteqwaveeq}
		\left(\frac{1}{r}\partial_r\right)^{\frac{n-2}{2}}r^{n-2}\M f(x,r)=\frac{2\gamma_n}{\pi n}\int_0^r\frac{u(x,t)}{\sqrt{r^2-t^2}}\d{t}
	\end{equation}
	and
	\begin{equation}
		\label{eq:dvsolabelinteqwaveeq}
		\left(\frac{1}{r}\partial_r\right)^{\frac{n-2}{2}}r^{n-2}D_v\M f(x,r)=\frac{2\gamma_n}{\pi n}\int_0^r\frac{D_v u(x,t)}{\sqrt{r^2-t^2}}\d{t}.
	\end{equation}
\end{prop}
\begin{proof}
	We give a proof for \eqref{eq:solabelinteqwaveeq}. Relation \eqref{eq:dvsolabelinteqwaveeq} can be proved analogously.
	
	First, we apply integration by parts to obtain for $t>0$
	\begin{equation*}
		\int_0^t\frac{r}{\sqrt{t^2-r^2}}\left(\frac{1}{r}\partial_r\right)^{\frac{n-2}{2}}\left(r^{n-2}\M f(x,r)\right)\d{r}=\int_0^t \sqrt{t^2-r^2}\partial_r\left(\frac{1}{r}\partial_r\right)^{\frac{n-2}{2}}\left(r^{n-2}\M f(x,r)\right)\d{r},
	\end{equation*}
	where we used that $\M f(x,\cdot)$ has compact support. Hence, the Leibniz-rule for integrals and \eqref{eq:solwaveeqeven2} imply \[u(x,t)=\frac{n}{\gamma_n}\int_0^t\frac{t}{\sqrt{t^2-r^2}}\partial_r\left(\frac{1}{r}\partial_r\right)^{\frac{n-2}{2}}\left(r^{n-2}\M f(x,r)\right)\d{r},\]
	and therefore \[\frac{1}{\sqrt{t}}\frac{\gamma_n}{n}u(x,\sqrt{t})=\int_0^t\frac{1}{\sqrt{t-r'}2\sqrt{r'}}\left.\partial_r\left(\frac{1}{r}\partial_r\right)^{\frac{n-2}{2}}\left(r^{n-2}\M f(x,r))\right)\right\vert_{r=\sqrt{r'}}\d{r'}\] by substituting $r'=r^2$. Now, we make use of formula \eqref{eq:solabelinteq} and the relation $\Gamma(\frac{1}{2})=\sqrt{\pi}$ to deduce for $r>0$ \[\frac{1}{2\sqrt{r'}}\left.\partial_r\left(\frac{1}{r}\partial_r\right)^{\frac{n-2}{2}}\left(r^{n-2}\M f(x,r))\right)\right\vert_{r=\sqrt{r'}}=\frac{\gamma_n}{\pi n}\left.\left(\frac{d}{dr}\int_0^r \frac{u(x,\sqrt{t'})}{\sqrt{t'}\sqrt{r-t'}}\d{t'}\right)\right\vert_{r=\sqrt{r'}}.\] Next, multiplying both sides with $2\sqrt{r'}$ and applying the chain rule lead to \[\partial_r\left(\frac{1}{r}\partial_r\right)^{\frac{n-2}{2}}r^{n-2}\M f(x,r)=\frac{\gamma_n}{\pi n}\frac{d}{dr}\int_0^{r^2}\frac{u(x,\sqrt{t'})}{\sqrt{t'}\sqrt{r^2-t'}}\d{t'}=\frac{2\gamma_n}{\pi n}\frac{d}{dr}\int_0^r\frac{u(x,t)}{\sqrt{r^2-t^2}}\d{t},\] where we substituted $t$ with $\sqrt{t'}$ in the last step. Finally, integrating both sides from zero to $r$ and using Lemma \ref{lem:relsopspmean}, we see that \eqref{eq:solabelinteqwaveeq} holds.\qedhere
\end{proof}
\begin{corollary}
	Under the assumptions of Proposition \ref{prop:solabelinteqwaveeqeven}, we have
	\begin{equation}
		\label{eq:solabelinteqwaveeq2}
		\left(\frac{1}{r}\partial_r\right)^{n-2}r^{n-2}\M f(x,r)=\frac{2\gamma_n}{\pi n}\int_0^r\frac{\left(\partial_t t^{-1}\right)^{\frac{n-2}{2}}u(x,t)}{\sqrt{r^2-t^2}}\d{t}
	\end{equation}
	and
	\begin{equation}
		\label{eq:dvsolabelinteqwaveeq2}
		\left(\frac{1}{r}\partial_r\right)^{n-2}r^{n-2}D_v\M f(x,r)=\frac{2\gamma_n}{\pi n}\int_0^r\frac{\left(\partial_t t^{-1}\right)^{\frac{n-2}{2}}D_v u(x,t)}{\sqrt{r^2-t^2}}\d{t}.
	\end{equation}
\end{corollary}
\begin{proof}
	Again, we only show \eqref{eq:solabelinteqwaveeq2}. Applying integration by parts on the right-hand side in \eqref{eq:solabelinteqwaveeq} gives the sum of two boundary terms plus the integral \[\frac{2\gamma_n}{\pi n}\int_0^r\sqrt{r^2-t^2}\partial_tt^{-1}u(x,t)\d{t}.\] From solution formula \eqref{eq:solwaveeqeven} we see that $t^{-1}u(x,t)=0$ for sufficient small $t>0$ and therefore both boundary terms equal zero. Then, using the integral rule of Leibniz implies \[\frac{1}{r}\partial_r\left(\frac{1}{r}\partial_r\right)^{\frac{n-2}{2}}r^{n-2}\M f(x,r)=\frac{2\gamma_n}{\pi n}\int_0^r\frac{\left(\partial_t t^{-1}\right)u(x,t)}{\sqrt{r^2-t^2}}\d{t}.\] The remaining claim follows by applying above arguments inductively.\qedhere
\end{proof}
\subsection{Remaining lemmas}
\begin{lemma}
	\label{lem:intinvsphericalmeanop}
	Let $a,b,c,d>0$, $a\neq b$ and $\max\set{a,b}<c<d$. Then, the following integral can be evaluated to
	\begin{equation}
		\label{eq:intinvsphericalmeanop}
		\int_c^d\frac{1}{x\sqrt{x^2-a^2}\sqrt{x^2-b^2}}\d{x}=F(d)-F(c),
	\end{equation}
	where $F(x)$ is defined by the term
	\begin{equation*}
		\frac{1}{2ab}\log\left(\left.\left(\sqrt{\frac{x^2-\max\set{a,b}^2}{x^2-\min\set{a,b}^2}}+\frac{\max\set{a,b}}{\min\set{a,b}}\right)\right/\left(\frac{\max\set{a,b}}{\min\set{a,b}}-\sqrt{\frac{x^2-\max\set{a,b}^2}{x^2-\min\set{a,b}^2}}\right)\right),
	\end{equation*} 
	for $x>\max\set{a,b}$, i.e., $F$ is an indefinite integral of the integrand in \eqref{eq:intinvsphericalmeanop}. Moreover, \[\int_{\max\set{a,b}}^\infty\frac{1}{x\sqrt{x^2-a^2}\sqrt{x^2-b^2}}\d{x}=\frac{1}{2ab}\log\left(\left.\left(1+\frac{\max\set{a,b}}{\min\set{a,b}}\right)\right/\left(\frac{\max\set{a,b}}{\min\set{a,b}}-1\right)\right).\]
\end{lemma}
\begin{proof}
	Assume, without loss of generality, that $a<b$. First, using the subsitution $x=\sqrt{u^2+a^2}$ gives \[\int_{u(c)}^{u(d)} \frac{1}{(u^2+a^2)\sqrt{u^2-(b^2-a^2)}}\d{u}.\] Next, we substitute $u$ with $\sqrt{b^2-a^2}\sec(v)$ to obtain \[\int_{v(u(c))}^{v(u(d))}\frac{\sec(v)}{(b^2-a^2)\sec(v)^2+a^2}\d{v}=\int_{v(u(c))}^{v(u(d))}\frac{\cos(v)}{b^2-\sin(v)^2a^2}\d{v},\] where we inserted the relation $\sec(v)=\frac{1}{\cos(v)}$ and used the trigonometric identity $\cos(v)^2=1-\sin(v)^2$ afterwards. Thus, the final substitution $w=\sin(v)$ leads to the standard integral \[\frac{1}{a^2}\int_{w(v(u(c)))}^{w(v(u(d)))}\frac{1}{\frac{b^2}{a^2}-w^2}\d{w}=\frac{1}{2ab}\left(\int_{w(v(u(c)))}^{w(v(u(d)))}\frac{1}{w+\frac{b}{a}}+\frac{1}{\frac{b}{a}-w}\d{w}\right).\] Finally, using the relation $\sin(\mathrm{arcsec}(x))=\frac{\sqrt{x^2-1}}{x}$ and inserting the upper and lower limit into the above integral yield the claimed identity.\\ The second statement is a consequence of the first integral identity, since $F(\max\set{a,b})=0$ and the limit $\lim_{d\to\infty} F(d)$ coincides with the right term in the second equality in Lemma \ref{lem:intinvsphericalmeanop}.\qedhere
\end{proof}
\begin{lemma}
	\label{lem:intinvbtneumanneven}
	Let $a,b,c,d\in\R$, $a<b$, $c\in\setcompl{[a,b]}$, $d\leq a$ and $c\neq d$. Then, we have the following identities:
	\begin{enumerate}
		\item
		\begin{equation}
			\begin{aligned}
				\int_a^b &\frac{x}{(x^2-c^2)\sqrt{x^2-d^2}}\d{x}\\
				&=\frac{1}{\sqrt{\abs{c^2-d^2}}}\begin{cases}
				\arctan\left(\frac{\sqrt{b^2-c^2}}{\sqrt{d^2-c^2}}\right)-\arctan\left(\frac{\sqrt{a^2-c^2}}{\sqrt{d^2-c^2}}\right),&\quad\text{c<d},\\
				\frac{1}{2}\log\left(\abs{\frac{\sqrt{c^2-d^2}-\sqrt{b^2-d^2}}{\sqrt{c^2-d^2}+\sqrt{b^2-d^2}}\frac{\sqrt{c^2-d^2}+\sqrt{a^2-d^2}}{\sqrt{c^2-d^2}-\sqrt{a^2-d^2}}}\right),&\quad\text{else.}
			\end{cases}
			\end{aligned}
		\end{equation}
		\item For $c>d$ we have
		\begin{equation}
			\lim_{\varepsilon\searrow 0} \frac{\sqrt{c^2-d^2}-\sqrt{(c-\varepsilon)^2-d^2}}{\sqrt{c^2-d^2}+\sqrt{(c-\varepsilon)^2-d^2}}\frac{\sqrt{(c+\varepsilon)^2-d^2}+\sqrt{c^2-d^2}}{\sqrt{(c+\varepsilon)^2-d^2}-\sqrt{c^2-d^2}}=1.
		\end{equation}
	\end{enumerate}	
\end{lemma}
\begin{proof}
	\begin{enumerate}[wide=\parindent,label=(\roman*)]
		\item We first substitute $x$ with $\sqrt{u^2+d^2}$ to obtain the integral \[\int_{u(a)}^{u(b)} \frac{1}{u^2+d^2-c^2}\d{u}.\] If $c<d$, then using the subsitution $u=\sqrt{d^2-c^2}v$ gives
		\begin{equation*}
			\frac{\sqrt{d^2-c^2}}{d^2-c^2}\int_{v(u(a))}^{v(u(b))}\frac{1}{1+v^2}\d{v}=\frac{1}{\sqrt{d^2-c^2}}\left(\arctan\left(\frac{\sqrt{b^2-c^2}}{\sqrt{d^2-c^2}}\right)-\arctan\left(\frac{\sqrt{a^2-c^2}}{\sqrt{d^2-c^2}}\right)\right).
		\end{equation*}
		If $d<c$, then \[\int_{u(a)}^{u(b)} \frac{1}{u^2-(c^2-d^2)}\d{u}=\frac{1}{2\sqrt{c^2-d^2}}\int_{u(a)}^{u(b)}\frac{1}{u-\sqrt{c^2-d^2}}-\frac{1}{u+\sqrt{c^2-d^2}}\d{u}.\] Hence, the above integral equals
		\[\frac{1}{2\sqrt{c^2-d^2}}\log\left(\frac{\sqrt{c^2-d^2}-\sqrt{b^2-d^2}}{\sqrt{c^2-d^2}+\sqrt{b^2-d^2}}\frac{\sqrt{c^2-d^2}+\sqrt{a^2-d^2}}{\sqrt{c^2-d^2}-\sqrt{a^2-d^2}}\right)\] for $b<c$ and
	\[\frac{1}{2\sqrt{c^2-d^2}}\log\left(\frac{\sqrt{b^2-d^2}-\sqrt{c^2-d^2}}{\sqrt{c^2-d^2}+\sqrt{b^2-d^2}}\frac{\sqrt{c^2-d^2}+\sqrt{a^2-d^2}}{\sqrt{a^2-d^2}-\sqrt{c^2-d^2}}\right)\] for $b>c$.
		\item Applying L'Hospital's rule gives
		\begin{align*}
			\lim_{\varepsilon\searrow 0} \frac{\sqrt{c^2-d^2}-\sqrt{(c-\varepsilon)^2-d^2}}{\sqrt{(c+\varepsilon)^2-d^2}-\sqrt{c^2-d^2}}=\lim_{\varepsilon\searrow 0}\frac{c-\varepsilon}{\sqrt{(c-\varepsilon)^2-d^2}}\frac{\sqrt{(c+\varepsilon)^2-d^2}}{c+\varepsilon}=1,
		\end{align*}
		which implies the second statement in Lemma \ref{lem:intinvbtneumanneven}.\qedhere
	\end{enumerate}
\end{proof}
\begin{lemma}
	\label{lem:fsinvbtneumanneven}
	Let $c>0$ be a positive number.
	\begin{enumerate}
		\item Define \[g_{\varepsilon}\colon\R\to\R^{+}\colon t\mapsto\begin{cases}
			\log\left(\frac{\sqrt{(c+\varepsilon)^2-t^2}+\sqrt{c^2-t^2}}{\sqrt{(c+\varepsilon)^2-t^2}-\sqrt{c^2-t^2}}\right),\quad &t\in[c-\varepsilon,c],\\
			0,\quad&\text{else,}
		\end{cases}\] for $\varepsilon>0$. Then, $g_\varepsilon$ is a monotonically decreasing on $[c-\varepsilon,c]$ and $\abs{g_\varepsilon}\leq\log(6+\sqrt{2})$.
		\item Let $0<\varepsilon<c$ and
			\begin{align*}
				h_{\varepsilon}\colon\R&\to\R^{+}\\
				t&\mapsto\begin{cases}
			\log\left(\frac{\sqrt{c^2-t^2}-\sqrt{(c-\varepsilon)^2-t^2}}{\sqrt{c^2-t^2}+\sqrt{(c-\varepsilon)^2-t^2}}\frac{\sqrt{(c+\varepsilon)^2-t^2}+\sqrt{c^2-t^2}}{\sqrt{(c+\varepsilon)^2-t^2}-\sqrt{c^2-t^2}}\right),\quad &t\in[0,c-\varepsilon],\\
			0,\quad&\text{else.}
		\end{cases}
			\end{align*}
			Then, $h_\varepsilon$ is a monotonically increasing on $[0,c-\varepsilon]$ and $\abs{h_\varepsilon}\leq\log(6+\sqrt{2})$.
	\end{enumerate}
	\begin{proof}
		\begin{enumerate}[wide=\parindent,label=(\roman*)]
			\item \label{item:idinvtbneumann1}By differentiating the inner fraction in $g_\varepsilon$, one easily verifies that the derivative is smaller than zero. Hence, from the monotonicity of the logarithmic function we imply that $g_\varepsilon$ is a monotonically decreasing on $[0,c-\varepsilon]$. Furthermore, note that the inner fraction is greater than or equal to one for $t\leq c$ and \[g_\varepsilon(c-\varepsilon)=\log\left(\frac{\sqrt{4c}+\sqrt{2c-\varepsilon}}{\sqrt{4c}-\sqrt{2c-\varepsilon}}\right).\] Therefore, using that $\varepsilon\mapsto\frac{\sqrt{4c}+\sqrt{2c-\varepsilon}}{\sqrt{4c}-\sqrt{2c-\varepsilon}}$ is decreasing for $0\leq\varepsilon\leq 2c$, the first statement holds.
			\item Differentiating the inner function in $h_\varepsilon$ for $t<c-\varepsilon$ with the product rule and a subsequent factorization lead to the expression \[2t\frac{G(t)}{\sqrt{c^2-x^2}}\left(\frac{1}{\sqrt{(c-\varepsilon)^2-t^2}}-\frac{1}{\sqrt{(c+\varepsilon)^2-t^2}})\right)> 0,\] where $G(t)$ denotes the inner function in $h_\varepsilon$. This implies the monotonicity of $h_\varepsilon$ on $[0,c-\varepsilon]$. Then, the same arguments as in \ref{item:idinvtbneumann1} show the remaining statement.\qedhere
		\end{enumerate}
	\end{proof}
\end{lemma}

\end{document}